\documentclass[a4paper,reqno,11pt]{amsart}
\usepackage{txfonts}
\usepackage{}
\usepackage{amsfonts}
\usepackage{amssymb}
\usepackage{amsfonts,amsmath,amsthm,amssymb,stmaryrd}

\newtheorem{Theorem}{Theorem}[section]
\newtheorem{Lemma}{Lemma}[section]
\newtheorem{Proposition}{Proposition}[section]
\newtheorem{Remark}{Remark}[section]

\newtheorem{Definition}{Definition}[section]
\numberwithin{equation}{section}
\usepackage[left=1 in, right=1 in,top=1 in, bottom=1 in]{geometry}
\def\Xint#1{\mathchoice
{\XXint\displaystyle\textstyle{#1}}%
{\XXint\textstyle\scriptstyle{#1}}%
{\XXint\scriptstyle\scriptscriptstyle{#1}}%
{\XXint\scriptscriptstyle\scriptscriptstyle{#1}}%
\!\int}
\def\XXint#1#2#3{{\setbox0=\hbox{$#1{#2#3}{\int}$ }
\vcenter{\hbox{$#2#3$ }}\kern-.6\wd0}}

\begin{document}
\title[Regularity of weak solutions ]{\bf Regularity of weak solutions to a  certain class of   parabolic system}
\author{Zhong Tan}
\address{School of Mathematical Sciences and Fujian Provincial Key Laboratory on Mathematical Modeling and Scientific Computing, Xiamen University,
Xiamen, Fujian , 361005 , P. R. China}
\email[Z. Tan]{tan85@xmu.edu.cn}
\author{Jianfeng Zhou}
\address{School of Mathematical Sciences, Peking University, Beijing 100871, China}
\email[J. Zhou]{jianfengzhou@pku.edu.cn}
\thanks{Corresponding author:Jianfeng Zhou, jianfengzhou@pku.edu.cn}
\thanks{
This work was supported by the National Natural Science Foundation
 of China (No. 11271305, 11531010).}
\begin{abstract}
We study  the regularity of weak solutions to a certain class of second order parabolic system under the only assumption of continuous coefficients. By using the $A-$caloric approximation argument, we claim that the weak solution $u$ to such system is locally H\"{o}lder continuous with any  exponent $\alpha\in(0,1)$ outside  a singular set with zero parabolic measure. In particular, we prove that the regularity point in $Q_T$ is an open set with full measure, and we obtain a general criterion for a weak solution to be regular in the neighborhood of a given point. Finally, we deduce the fractional time and fractional space differentiability of $D u$, and at this stage, we obtain the Hausdorff dimension of singular set of $u$.
\bigbreak
\end{abstract}

\maketitle
\section{Introduction}
Let  $\Omega\subset \mathbb{R}^n$  ($n\geq2$) be a bounded domain, the aim of this work is to give a study of regularity properties of weak solution to the following inhomogeneous parabolic system
\begin{equation}\label{1.1}
\partial_t u-{\rm div}\ a(z,u,Du)=b(z,u,Du),
\end{equation}
with $z=(x,t)\in\Omega\times(-T,0)\equiv Q_T$, and $T>0$. $a(\cdot):Q_T\times \mathbb{R}^N\times \mathbb{R}^{Nn}\rightarrow \mathbb{R}^{Nn}$, $u:Q_T\rightarrow \mathbb{R}^N$, $b:Q_T\times \mathbb{R}^N\times \mathbb{R}^{Nn}\rightarrow \mathbb{R}^N$, $N\in \mathbb{Z}^+$, $N\geq1$.
In general, the solution of parabolic systems (\ref{1.1})  can not be expected to be regular everywhere on the domain, even the homogeneous case
\begin{equation}\label{1.2}
\partial_t u-{\rm div}\ a(z,u,Du)=0.
\end{equation}
It is worth to note that everywhere regularity can be obtained only with special structure on $a(z,u,Du)$ such as the evolutionary $p-$Laplacian system
\begin{equation*}
\partial_t u-{\rm div}(|Du|^{p-2}Du)=0,
\end{equation*}
for the regularity problem was settled by the fundamental contributions of Dibenedetto's and Friedman \cite{A19, cs12,cs13}, otherwise it fails in general
see \cite{le24,le25,le27} for example.

However, one can expect partial regularity results, this is regularity away from a singular set that is in some sense small. The partial regularity for general parabolic  (\ref{1.2}) was a longstanding open problem until it was solved by Duzaar and Mingione \cite{D17},  Duzaar, Mingione and Steffen \cite{D18}, C. Scheven \cite{d23} and also Duzaar et al. \cite{D16, D3,D4}, their proofs are based on the $A-$caloric approximation method to the parabolic setting. Subsequently, Scheven \cite{d23} derived an analogous result for the subquadratic case of (\ref{1.2}). Moreover, Baroni \cite{bar} have showed the continuity of the gradient $Du$ while only assuming the Dini continuity of $a(\cdot,\cdot,Du)$.  Under the assumption of continuous coefficients, B\"{o}gelein-Duzzar-Mingione \cite{bfm} proved a partial H\"{o}lder continuity results for (\ref{1.2}) with polynomial growth.
When considering the boundary regularity of the parabolic system, the same authors \cite{D3,D4} have showed that almost every parabolic boundary point is a H\"{o}lder continuity point for $Du$. There have been many research articles on the regularity of weak solution to parabolic system, e.g., \cite{bfm1,ct,bfm6,bfm29,dm32,tz} and the reference therein.

The above result for parabolic problems are analogous of results of elliptic case (cf. \cite{m2006}), the application of the so called harmonic approximation   to prove regularity theorems goes back to Simon \cite{le22,le23} and Duzaar et al. \cite{le9, le10}. Related results for problems with
continuous coefficients, Campanato \cite{sch4} (see also \cite{mp4}) derived the H\"{o}lder continuity of the solutions of some nonlinear elliptic system in $\mathbb{R}$. In higher dimensions cases, Foss-Mingione \cite{w12} proved the partial H\"{o}lder continuity for solutions to elliptic system. The proof relies upon an iteration scheme of a decay estimate for a new type of excess functional measuring the oscillations in the solution and its gradient. Afterwards, Beck \cite{beck}  showed the boundary regularity of elliptic system with Dirichlet condition. When considering the Dini continuous coefficients, Duzzar-Gastel \cite{w8} presented a general low-order partial regularity theory.
In particular, for the system with variable exponent $p(x)$, Habermann \cite{h2013} (see also \cite{l7}) derived the partial H\"{o}lder continuity for
weak solution to a nonlinear problem with continuous growth exponent. For more details, one can also refer \cite{D2,w2,mf10,sch10,sch14,sch19,mf32} and
the reference therein.

Turning to the  technically more challenging case of (\ref{1.1}), as far as we are aware, there has been no previous work addressing partial regularity of weak solution $u$ to (\ref{1.1}) with continuous coefficients available in the literature yet (cf. \cite{bfm} for the homogeneous case (\ref{1.2})).  Thus, in present paper,  we aim to  fill a gap in the partial regularity theory of quasi-linear parabolic system (\ref{1.1}). This turns out to be a challenging task, since the nonhomogeneous term $b(z,u,Du)$ will lead to several new difficulties:
\begin{enumerate}
\item When establish the Poincar\'{e} inequality in Section \ref{se4}, we are not able to obtain (\ref{4.96}) directly, since we can not use the
zero-boundary condition on  $\partial B_{\rho}$ for any  $B_{\rho}\subset\Omega$. In order to avoid this flaw, some  iteration argument will be
introduced;
\item For prove the  Caccioppoli's inequality  (\ref{3.1}), the key point is that, bound $b(z,u,Du)$ in terms of $Du-Dl$ or $u-l$ ($l$ be an affine
function defined in later). However, one can not use the inequality $l(z)\leq l(z_0)+Dl\leq M$ directly for a.e.
$z\in Q_{\rho}$, $\rho\leq1$ and $M\geq1$ be a constant. Otherwise, the constant after (\ref{5.13}) depends on $M$ with $M=H\lambda$ (see (\ref{5.7})
and (A$_j$)). As a consequence, all constants in Lemma \ref{le5.3} depend on $\lambda$ so that the estimates could blow up during the iteration process.
At this stage, we shall use a weighted Sobolev interpolation inequality (cf. \cite{flc2,ckn2,bc2}): for suitable function
$w(\cdot):\Omega\longrightarrow\mathbb{R}_+$ satisfies
\begin{equation*}
\sup_{\mbox{\tiny
$\begin{array}{c}
x,y\in\Omega \\
|x-y|<w(y)
\end{array}$}}
\left|\ln\frac{w(x)}{w(y)}\right|<+\infty,
\end{equation*}
and for any function $v$, $k\in \mathbb{N}_0$, $r\in \mathbb{N}$, $s\in \mathbb{R}$, $p_0\in[1,+\infty)$, $p_1,q\in (1,+\infty)$,
 if $\partial^rv\in L^{p_1}_s(\Omega)$, $v\in L^{p_0}_{\alpha_0-r}(\Omega)$, $k<r$, $p_0\leq p_1\leq q$, $\frac1q>\frac{1}{p_1}-\frac{r-k}{n}$,
 $\alpha_0=s-n(\frac{1}{p_0}-\frac{1}{p_1})$,  $\alpha=s-n(\frac{1}{q}-\frac{1}{p_1})$ and
\begin{equation*}
\theta=\frac{\frac{1}{p_0}-\frac1q+\frac kn}{\frac{1}{p_0}-\frac{1}{p_1}+\frac rn},
\end{equation*}
then, there holds
\begin{equation}\label{1.3}
\|v\|_{W^{k,q}_{\alpha-(r-k)}(\Omega)}\leq c(\|\partial^rv\|^{\theta}_{L_s^{p_1}(\Omega)}\|v\|^{1-\theta}_{L^{p_0}_{\alpha_0-r}(\Omega)}
+\|v\|_{L^{p_0}_{\alpha_0-r}(\Omega)}),
\end{equation}
where $c$ depends on $p_0$, $p_1$, $n$, $q$, $r$, $k$. Here, we have defined
\begin{align*}
\|v\|_{W^{s_1,s_2}_{s}(\Omega)}:=&\sum_{|\alpha|\leq s_1}\|w(x)^{s+|\alpha|-s_1}\partial^{\alpha}v\|_{L^{s_2}(\Omega)},\\
\|v\|_{L^{k_1}_s(\Omega)}:=&\|w(x)^sv\|_{L^{k_1}(\Omega)},
\end{align*}
with $s_1\in \mathbb{N}$, $k_1,s_2\geq1$ and $s\in \mathbb{R}$.
\end{enumerate}
The main result of the present paper is stated as following
\begin{Theorem}\label{th1.1}
 Let $p\geq 2$ and $u\in C^0(-T,0;L^2(\Omega;\mathbb{R}^N))\cap L^p(-T,0;W^{1,p}(\Omega;\mathbb{R}^N))$ be a weak solution of the parabolic systems (\ref{1.1}) in $ Q_T$ under the assumptions (\ref{2.01})-(\ref{2.5}). Then, there exists an open subset $Q_0\subset Q_T $ such that
 \begin{equation*}
 |Q_T\setminus Q_0|=0 \ \ \  and \ \ \  u\in C^{0;\alpha,\alpha/2}(Q_0;\mathbb{R}^N)
 \end{equation*}
 for every $\alpha\in(0,1)$. Moreover, we have the singular set satisfies $Q_T\setminus Q_0\subset\Sigma_1\cup\Sigma_2$, where
\begin{align*}
\Sigma_1:=&\left\{z_0\in Q_T:\liminf_{\rho\longrightarrow\ 0}\Xint-_{Q_\rho (z_0)}\mid Du-(Du)_{z_0,\rho}\mid^p dz>0\right\},\\
\Sigma_2:=&\left\{z_0\in Q_T:\limsup_{\rho\longrightarrow\ 0} |(Du)_{z_0,\rho}|=\infty\right\}.
\end{align*}
\end{Theorem}
The main technique we have used in the proof of Theorem \ref{th1.1} is the $A-$caloric approximation lemma. Here, $A$ is a bilinear form on
$\mathbb{R}^{Nn}\times\mathbb{R}^{Nn}$ with constant coefficients. If $A$ satisfies certain growth and ellipticity conditions, then the weak solution $h$
to (\ref{5.05}) is $A-$caloric and have nice decay properties. In order to look for such `good' function, we shall use the $A-$caloric approximation lemma
(cf. Lemma \ref{le2.3}), from which, we can transfer the property of $A-$caloric to some `bad' function (target function).
When applying the $A-$caloric approximation lemma, we need to pay attention to three necessary conditions: i) the target function is bounded from above on
the scale of $L^2$-norm and $L^p$-norm; ii) the target function satisfies a linearized system; iii) the target function satisfies the smallness condition
in the sense of distribution. To satisfy these three conditions, we will establish the Caccioppoli inequality and linearize the system (\ref{1.1}) in Sec.
\ref{se3} and Sec. \ref{se5}, respectively. On the other hand, with the help of linearization lemma (cf. Lemma \ref{le5.1}), we would like to show
$w:=u-l_{\rho}$ approximately solves
\begin{equation*}
\Xint-_{Q_{\frac{\rho}{2}}(z_0)}w\cdot\varphi_t-(\partial_Fa(z_0,l_{\rho}(x_0),Dl_{\rho})Dw,D\varphi)dz=0
\end{equation*}
for all $\varphi\in C_0^{\infty}(Q_{\rho}(z_0);\mathbb{R}^N)$. Here, $l_{\rho}: B_{\rho} \longrightarrow \mathbb{R}^N$ be the unique time independent affine map minimizing $l\mapsto \Xint-_{Q_{\rho}(z_0)}|u-l|^2dz$. At this stage,  by the $A-$ caloric approximation lemma, then we can establish smallness
of the first order excess functional
\begin{equation*}
  \Phi(z_0,l,\rho):=\Xint-_{Q_{\rho}(z_0)}\left|\frac{u-l}{\rho}\right|^2+\left|\frac{u-l}{\rho}\right|^pdz.
\end{equation*}
From which, we are able to measure the oscillation in $u$ with respect to an affine mapping. Moreover, in order to  provide a bilinear form that satisfies the growth and ellipticity bounds needed to apply the $A-$caloric approximation lemma, we may need  the integral estimate on
intrinsic  cylinders, that is, parabolic cylinders stretched according to the size of the solution $u$ itself. The rough asymptotic is given by
\begin{equation*}
\left|\frac{1}{|Q_{\rho}^{(\lambda)}(z_0)|}\int_{Q_{\rho}^{(\lambda)}(z_0)} udz\right|\approx\lambda,
\end{equation*}
 with $Q_{\rho}^{(\lambda)}:B_\rho(x_0)\times(t_0-\lambda^{2-p}\rho^2,t_0)$.

According to  Theorem \ref{th1.1}, we immediately deduce that
\begin{equation}\label{x2.3}
\omega(d(z,z_0)^2+|u-u_0|^2)\lesssim d(z,z_0)+1_{K},
\end{equation}
where $K:=\Sigma_1\cup\Sigma_2$ and  $1_{K}=1$ for $x\in K$, otherwise,  $1_{K}=0$. Then we have the following result.
\begin{Theorem}\label{th1.2}
Let the all assumptions in Theorem \ref{th1.1} be verified and $p=2$. Then, for any $\theta\in (0,\frac{1}{3})$, $\alpha\in (0,\frac{1}{2}]$, there holds
$Du \in W^{\alpha,\theta;2}_{loc}(Q_T)$. Furthermore, the   singular set $\Sigma_1$ and $\Sigma_2$ in Theorem \ref{th1.1} satisfying \begin{equation}\label{hd1.7}
 \text{dim}\mathcal {H}(\Sigma_1)\leq n+2-2\gamma,\ \ \ \ \   \text{dim}\mathcal {H}(\Sigma_2)\leq n+2-2\gamma,
\end{equation}
where $\gamma\leq\alpha,$ $ \frac{\gamma}2\leq\theta$.
\end{Theorem}

The rest of paper is organised as follows. First of all, in Sec. \ref{se2} we state some assumption of the structure function $a(\cdot)$ and the inhomogeneity term $b(\cdot)$. Moreover, we present some notation, definition of weak solution to (\ref{1.1}), and some useful lemma which will be used in our proof. Next, we provide some preliminary  material in Sec. \ref{se3} and Sec. \ref{se4}, which will be quite useful in the proof of main result. The first step of our proof is to establish a Caccioppoli's type  inequality.   Subsequently, we establish a Poincar\'{e} type inequality  in Section \ref{se4}, which is useful to show the boundness of $|Dl|$, with $l$ be an affine function. In Sec. \ref{se5}, we first provide a linearization strategy for context, we show a decay estimate of $\Phi_{\lambda_j}(\vartheta^j\rho)$,  and then obtain a Campanato type estimate. This, combined with
a standard argument implies the Theorem \ref{th1.1}. Finally,  in Sec. \ref{se6},   we derive the fractional time and space differentiability of $Du$,
from which, we  estimate  the Hausdorff dimension  of singular set  of weak solution $u$ to (\ref{1.1}).
\section{preliminaries}\label{se2}
\subsection{Notation} Let $x_0\in \mathbb{R}^n$, $t_0\in \mathbb{R}$, $z_0=(x_0,t_0)$, we denote
\begin{equation*}
  B_\rho(x_0):=\{x\in \mathbb{R}^n:|x-x_0|<\rho\}
\end{equation*}
as an open ball in $\mathbb{R}^n$, and let
\begin{equation*}
  Q_\rho(x_0):={B_\rho(x_0)\times(t_0-\rho^2,t_0)\equiv B_\rho(x_0)\times\Lambda_\rho(t_0)}
\end{equation*}
as a cylinder in $\mathbb{R}^{n+2}$. Let $B_\rho(x_0)$, $Q_\rho(z_0)\subset Q_T$, and $f(x,t)$ integrable on $B_\rho(x_0)$ and  $Q_\rho(z_0)$, then the average integral of $f$ over $B_\rho(x_0)$ and  $Q_\rho(z_0)$ are defined by
\begin{equation*}
(f)_{x_0,\rho}=\Xint-_{B_\rho(x_0)} fdx=\frac{1}{|B_\rho(x_0)|}\int_{B_\rho(x_0)} fdx,
\end{equation*}
and
\begin{equation*}
(f)_{z_0,\rho}=\Xint-_{Q_\rho(z_0)} fdz=\frac{1}{|Q_\rho(z_0)|}\int_{Q_\rho(z_0)} fdz.
\end{equation*}
In what follows, we shall repeatly use the scaled parabolic cylinders of the form
\begin{equation*}
  Q_{\rho}^{(\lambda)}(z_0)=B_\rho(x_0)\times\Lambda_{\rho}^{(\lambda)}(t_0)
\end{equation*}
with radius $\rho>0$, scaling factor $\lambda>0$, and
\begin{equation*}
\Lambda_{\rho}^{(\lambda)}(t_0):=(t_0-\lambda^{2-p}\rho^2,t_0).
\end{equation*}
In particular, when $\lambda=1,$ then  $Q_{\rho}^{(1)}(z_0)\equiv Q_{\rho}(z_0)$, and $\Lambda_{\rho}^{(1)}(t_0)\equiv \Lambda_{\rho}(t_0)$. Furthermore, The parabolic metric is defined as usual by
\begin{equation*}
  d(z,z_0)=:max\left\{|x-x_0|,\sqrt{|t-t_0|}\right\}\approx\sqrt{|x-x_0|^2+|t-t_0|}.
\end{equation*}
Based on the parabolic metric, the space $C^{k;\alpha_1,\alpha_2}(Q_T)$ are those of functions $u\in C^k(Q_T)$ which  are $\alpha_1$-H\"{o}lder continuous in the space variables  $\alpha_2$-H\"{o}lder continuous in the time variables. More precisely,  we call $u\in C^{k;\alpha,\alpha/2}(\Omega_T;\mathbb{R}^N)$ ($k\geq0$ be an integer), if
\begin{equation*}
  u\in C^{k;\alpha,\alpha/2}(Q_T;\mathbb{R}^N):=\left\{v\in C^k(Q_T;\mathbb{R}^N): \sup_{z,z_0\in Q_T; z\neq z_0}\frac{|v(z)-v(z_0)|}{d(z,z_0)^\alpha}<\infty\right\}.
\end{equation*}
We said $u\in C_{loc}^{k;\alpha,\alpha/2}(Q_T;\mathbb{R}^N)$ if and only if for all $ A\subset Q_T$, there holds $u\in C^{k;\alpha,\alpha/2}(A;\mathbb{R}^N)$. Finally, we note that in the whole paper, we use the notation $(\cdot,\cdot)$ denote the inner product.

For $s\in[0,n+2]$ and $E\subset \mathbb{R}^{n+1}$, we define the (parabolic) Hausdorff measure:
\begin{equation*}
\mathcal {H}_{s}^{\delta}(E):=\inf\left\{\sum_{i=1}^{\infty}r_i^s:\ E\subset\bigcup_{i=1}^{\infty}Q_{r_i}(x_i,t_i),r_i\leq\delta\right\},
\ \ \   \mathcal {H}_{s}(E):=\sup_{\delta>0}\mathcal {H}_{s}^{\delta}(E).
\end{equation*}
From above, then the Hausdorff dimension  is usually defined by
\begin{equation*}
\text{dim}\mathcal {H}(E):=\inf\left\{s>0:\mathcal {H}_{s}(E)=0\right\}=\sup\left\{s>0:\mathcal {H}_{s}(E)=\infty\right\}.
\end{equation*}
Moreover, in this paper we use $D$ or $\nabla$ denotes the `gradient', and we will use the following natation:
\begin{align*}
&\tau^hv(x,t):=v(x,t+h)-v(x,t),\\
&\tau^h_{1,2}a(z,u,Du):=a((x,t+h),u(x,t+h),Du)-a(z,u,Du)\\
&\tau^h_{3}a(z,u,Du):=a(z,u,Du(x,t+h))-a(z,u,Du)\\
& \tau_h v(x,t)\equiv \tau_{h,i}v(x,t):=v(x+he_i,t)-v(x,t),\\
&\tau_h^{1,2}a(z,u,Du)\equiv\tau_{h,i}^{1,2}a(z,u,Du):=a((x+he_i,t),u(x+he_i,t),Du)-a(z,u,Du)\\
&\tau_h^{3}a(z,u,Du)\equiv\tau_{h,i}^{3}a(z,u,Du):=a(z,u,Du(x+he_i,t))-a(z,u,Du)\\
&\triangle_h v(x,t)\equiv \triangle_{h,i}v(x,t):=\frac{v(x+he_i,t)-v(x,t)}{h}.
\end{align*}
Here $e_i=(0,\cdots0,1_{i-th},0,\cdots,0)$, $i=1,\cdots,n$. Finally, let us recall the definition of parabolic fractional Sobolev space (refer to
\cite{le17} for details). We say $u\in L^2(Q_T)$
belongs to the fractional Sobolev space $W^{\alpha,\theta;2}(Q_T)$, $\alpha, \theta\in (0,1)$, if
\begin{equation*}
\int_{-T}^{0}\int_{\Omega}\int_{\Omega}\frac{|u(x,t)-u(y,t)|^2}{|x-y|^{n+2\alpha}}dxdydt+\int_{\Omega}  \int_{-T}^{0}  \int_{-T}^{0}
\frac{|u(x,t)-u(x,s)|^2}{|t-s|^{1+2\theta}}dtdsdx=:[u]_{\alpha,\theta;Q_T}<\infty.
\end{equation*}
\subsection{ Assumption on the structure function $a(\cdot)$ and $b(\cdot)$}
 In the following, we impose the condition on the structure function $a(z,u,F)$ and $b(z,u,F)$ for $p\geq 2$.
\begin{itemize}
  \item The growth condition
\begin{equation}\label{2.01}
  |a(z,u,F)|+(1+|F|)|\partial_F a(z,u,F)|\leq L(1+|F|)^{p-1},
\end{equation}
  with $(z,u,F)\in Q_T\times {\mathbb{R}^N} \times \mathbb{R}^{Nn}, L\geq 1$ be a constant.
  \item The ellipticity condition
\begin{equation}\label{2.02}
  \partial_F a(z,u,F)(\widetilde{F},\widetilde{F})\geq\nu (1+|F|)^{p-2}|\widetilde{F}|^2,
\end{equation}
for all $(z,u,F)\in Q_T\times \mathbb{R}^N \times \mathbb{R}^{Nn}$, $\widetilde{F}\in \mathbb{R}^{Nn}$, $0<\nu\leq 1\leq L$ be a constant.

Moreover,  we also need the following two continuity conditions:
 \item Continuity of lower order term
 \begin{equation}\label{2.03}
   |a(z,u,F)-a(z_0,u_0,F)|\leq L \omega\left(d(z,z_0)^2+|u-u_0|^2\right)(1+|F|)^{p-1},
\end{equation}
\item  Continuity of higher order term
\begin{equation}\label{2.04}
 |\partial_F a(z,u,F)-\partial_F a(z,u,F_0)|\leq L\mu\left(\frac{|F-F_0|}{1+|F|+|F_0|}\right)(1+|F|+|F_0|)^{p-2},
\end{equation}
for all $z,z_0\in Q_T$, $u,u_0\in \mathbb{R}^N$ and $F,F_0\in \mathbb{R}^{Nn}$. Here, $\omega,\ \mu:[0,\infty)\rightarrow[0,1]$ are two bounded, concave, and non-decreasing functions satisfy
\begin{equation*}
  \lim_{s\rightarrow0}\mu(s)=\lim_{s\rightarrow0}\omega(s)=0.
\end{equation*}
The term  $b(z,u,F)$ satisfies  the following  growth conditions:
\item Controllable growth condition
\begin{equation}\label{2.5}
  |b(z,u,F)|\leq L(1+|F|)^{p-1}+|u|^{q_1},
\end{equation}
for all $(z,u,F)\in Q_T\times \mathbb{R}^N \times \mathbb{R}^{Nn}$ with $q_1\in[0, \frac{(n+2)(p-1)}{n}]$,  where the  upper bound of $q_1$ depends on the Ladyzhenskaya inequality.
\end{itemize}
\subsection{Definition of weak solution}
 Let $p\geq2$, we call $u\in C^0(-T,0;L^2(\Omega;\mathbb{R}^N))\cap L^p(-T,0;W^{1,p}(\Omega;\mathbb{R}^N))$ is a weak solution to (\ref{1.1}), if and only if the following identity
 \begin{equation}\label{2.1}
 \int_{Q_T} u\cdot\varphi_t-a(z,u,Du)\cdot D\varphi+b(z,u,Du)\cdot\varphi dz=0,
\end{equation}
holds for all $\varphi\in C_{0}^{\infty}(Q_T;\mathbb{R}^N)$.

From \cite{A19} (see also \cite{le17}) we recall the definition of the Steklov averages that allow us to restate (\ref{2.1}) in an equivalent
way. Let $v\in L^1(Q_T)$  and $0< h<T$, the Steklov averages $v_{h}$ and $v_{\bar{h}}$ are defined by
\begin{align*}
v_h(x,t):=\begin{cases}
 \frac{1}{h}\int_{t}^{t+h}v(x,s)ds& t\in (-T,-h), \\
0&t>-h,
\end{cases}
\end{align*}
and
\begin{align*}
v_{\bar{h}}(x,t):=\begin{cases}
 \frac{1}{h}\int_{t-h}^{t}v(x,s)ds& t\in (-T+h,0), \\
0&t<-T+h,
\end{cases}
\end{align*}
respectively, for all $t\in (-T,0)$. We note that if $v\in L^r(-T,0;L^q(\Omega))$ with $r,q\geq 1$, then $v_h\longrightarrow v$ in
$L^r(-T+\varepsilon,0;L^q(\Omega))$
as $h\longrightarrow 0$, for every $t\in (-T+\varepsilon,0)$ and $\varepsilon\in (0,T)$, and the same result holds for $v_{\bar{h}}$.

In virtue of  the convergence properties of the Steklov averages, then we have a  equivalent definition of weak solution to (\ref{1.1}):
\begin{Definition}\label{de2.2}
(A equivalent definition of weak solution).\  Let $2\leq p<\infty$ and $u_0\in L^2(\Omega)$. Then $u\in L^{\infty}(-T,0;L^2(\Omega))\cap L^p(-T,0;W^{1,p}(\Omega))$ is called a weak solution to (\ref{1.1}) if
\begin{equation}\label{hd2.2}
\int_{\Omega}\partial_tu_h\cdot\varphi +[a(z,u,Du)]_h\cdot D\varphi-[b(z,u,Du)]_h\cdot \varphi dx=0.
\end{equation}
holds for all $\varphi\in C_{0}^{\infty}(Q_T;\mathbb{R}^N)$.
\end{Definition}
Employing ($H_1$) and ($H_2$), we have the following result.
\begin{Lemma}\label{hdle2.1}
Let $2\leq p<\infty$, then there exists a constant $c=c(L,n,p)>0$  such that for any $F_1, F_2\in \mathbb{R}^{Nn}$, it holds that
\begin{equation*}
 c^{-1}(1+|F_1|^2+|F_2|^2)^{\frac{p-2}{2}}|F_1-F_2|^2\leq |a(z,u,F_1)-a(z,u,F_2)|^2\leq c(1+|F_1|^2+|F_2|^2)^{\frac{p-2}{2}}|F_1-F_2|^2.
\end{equation*}
\end{Lemma}
Next, the following lemma as an auxiliary tool will be heavily used (cf. \cite{A12}).
\begin{Lemma}\label{le2.1}
Let $A,B\in \mathbb{R}^k$, $k\geq 1$ and $\sigma> -1$, then there exists a constant $c=c(\sigma)$, such that
\begin{equation*}
c^{-1}(1+|A|+|B|)^\sigma \leq\int_{0}^{1}(1+|A+sB|)^\sigma ds \leq c(1+|A|+|B|)^\sigma.
\end{equation*}
\end{Lemma}
As a consequence, from Lemma \ref{le2.1} and (\ref{2.02}), it follows that the monotonicity of $a(z,u,\cdot)$:
\begin{equation}\label{hd2.5}
(a(z,u,F_1)-a(z,u,F_2))\cdot (F_1-F_2)\geq c(|F_1|+|F_2|)^{p-2}|F_1-F_2|^2,\ \ \   p\geq2,
\end{equation}
where $c=c(n,p,\nu)$.

In the next proposition we recall the parabolic version of the well known relation between Nikolski spaces and Fractional Sobolev spaces (cf. \cite{ok}).
\begin{Proposition}\label{pr2.1}
Let $u\in L^2(Q_T;\mathbb{R}^N)$, suppose
\begin{equation*}
\int_{Q'}|u(x,t+h)-u(x,t)|^2dxdt\leq c_1|h|^{2\theta},\quad \theta\in (0,1)
\end{equation*}
where $Q':=\Omega'\times (-T+\delta,-\delta)$ and $\Omega'\subset\subset \Omega$ for every $h\in\mathbb{R}$, such that $|h|\leq \min
\left\{\delta,1\right\}$ with $\delta\in (0,\frac{T}{8})$.Then there exists a constant $c'=c'(\theta,\gamma,\delta,\|u\|_{L^2(Q_T)})>0$ such that
\begin{equation*}
  \int_{\Omega'}\int_{-T+\delta}^{-\delta}\int_{-T+\delta}^{-\delta}\frac{|u(x,t)-u(x,s)|^2}{|t-s|^{1+2\gamma}}dtdsdx\leq c',
\end{equation*}
for all $\gamma\in (0,\theta)$. Furthermore, suppose that
\begin{equation*}
  \int_{Q'}|u(x+he_s,t)-u(x,t)|^2dxdt\leq c_2|h|^{2\theta},\ \ \ \ \ \ \theta\in (0,1)
\end{equation*}
for every $|h|\leq \min\{dist(\Omega',\partial\Omega),1\}$, $s\in\{1,\cdots,n\}$, with $\{e_s\}_{s=1}^n $ is the standard  basis of $\mathbb{R}^n$.
Then, for every $\tilde{\Omega}\subset\subset \Omega'$ there exists a constants $c''=c''(\delta,\theta,\gamma, c_2, dist(\Omega',\partial\Omega),dist(\partial\Omega',\tilde{\Omega}),\|u\|_{L^2(Q_T)} )$ such that
\begin{equation*}
\int_{-T+\delta}^{-\delta}\int_{\tilde{\Omega}}\int_{\tilde{\Omega}}\frac{|u(x,t)-u(y,t)|^2}{|x-y|^{n+2\gamma}}dxdydt\leq c'',
\end{equation*}
for all $\gamma\in (0,\theta)$.
\end{Proposition}
From Proposition \ref{pr2.1}, we can see that in order to prove the fractional differentiability of $Du$ in Theorem \ref{th1.2}, it is only need to prove
\begin{equation}\label{hd2.6}
\int_{Q'}|\tau^hDu|^2dxdt\leq c_1|h|^{\theta},
\end{equation}
for all $\theta\in (0,\frac{2}{3})$, and
\begin{equation}\label{hd2.7}
\int_{Q'}|\tau_hDu|^2dxdt\leq c_2|h|^{\gamma},
\end{equation}
for $\gamma\in (0,1].$

On the other hand, for estimate the Hausdorff dimension of singular set of $u$ defined in Theorem \ref{th1.1}, we shall use the following arguments (cf.
\cite{sec10,sec30}).
\begin{Lemma}\label{hdle2.3}
Let $\Omega\subset\mathbb{R}^n$ ($n\geq 2$), $v\in W^{\beta,\beta/2;2}_{loc}(Q_T;\mathbb{R}^N)$ with $\beta>0$ and $N\geq1$.  Let
\begin{align*}
  A:=&\left\{z_0\in Q_T:\liminf_{\rho\longrightarrow0}\Xint-_{Q_{\rho}(z_0)}|v-(v)_{z_0,\rho}|^2dz>0\right\},\\
  B:=&\left\{z_0\in Q_T:\limsup_{\rho\longrightarrow0}|(v)_{z_0,\rho}|=\infty\right\}.
\end{align*}
Then, there holds
\begin{equation*}
  \text{dim}\mathcal {H}(A)\leq n+2-2\beta,\ \ \ \ \ \   \text{dim}\mathcal {H}(B)\leq n+2-2\beta.
\end{equation*}
\end{Lemma}
\subsection{Minimizing affine function}
 Let $z_0\in \mathbb{R}^{n+2}$, $\rho$, $\lambda>0$. For a given function $u\in L^2(Q_{\rho}^{(\lambda)}(z_0);\mathbb{R}^N)$, we denote by $l_{z_0,\rho}^{(\lambda)}: \mathbb{R}^n\rightarrow \mathbb{R}^N$
the unique affine function (in space) minimizing
\begin{equation*}
l\mapsto \Xint-_{Q_{\rho}^{(\lambda)}(z_0)} |u-l|^2 dz,
\end{equation*}
amongst all affine function $l(z)=l(x)$ independent of $t$. We note that such a unique minimizing affine function exists and takes the form
\begin{equation}\label{2.2}
  l_{z_0,\rho}^{(\lambda)}(z)=\xi_{z_0,\rho}^{(\lambda)}+A_{z_0,\rho}^{(\lambda)}(x-x_0),
\end{equation}
 where $\xi_{z_0,\rho}^{(\lambda)}\in \mathbb{R}^N$, $A_{z_0,\rho}^{(\lambda)}\in \mathbb{R}^{Nn}$. A straightforward calculation shows that
 \begin{equation*}
 \Xint-_{Q_\rho^{(\lambda)}(z_0) }(u-l_{z_0,\rho}^{(\lambda)}) a(x)dz=0,
 \end{equation*}
for any $a(x)=\xi+A(x-x_0)$ with $\xi\in \mathbb{R}^N$, $A\in \mathbb{R}^{Nn}$. This implies in particular that
\begin{equation}\label{2.3}
  \xi_{z_0,\rho}^{(\lambda)}=(u)_{z_0,\rho}^{(\lambda)}\ \ \ \  and \ \  \ \ A_{z_0,\rho}^{(\lambda)}=\frac{n+2}{\rho^2} \Xint-_{Q_\rho^{(\lambda)}(z_0)} u\otimes(x-x_0)dz.
\end{equation}
Furthermore, we need the following argument, which can be proven analogously to \cite{p26}. For any $\xi\in \mathbb{R}^n$ and $ A\in\mathbb{R}^{Nn}$ there holds
\begin{equation}\label{2.4}
  |A_{z_0,\rho}^{(\lambda)}-A|^2\leq \frac{n(n+2)}{\rho^2}\Xint-_{Q_\rho^{(\lambda)}(z_0)} |u-\xi-A(x-x_0)|^2 dz.
\end{equation}
Finally, we introduce the following conclusion (cf. \cite{A9} Lemma 3.8), which provide a connection between the minimizing affine functions $l(z)$ and $l_{z_0,\rho}^{(\lambda)}(z)$.
\begin{Lemma}\label{le2.2}
Let $p\geq 2$, $Q_\rho^{(\lambda)}(z_0) \in \mathbb{R}^{n+2}$ with $z_0\in \mathbb{R}^{n+2}$ and $\rho,\lambda>0$ be a scaled parabolic cylinder and $u\in L^p(Q_\rho^{(\lambda)}(z_0);\mathbb{R}^N)$, and let $l:\mathbb{R}^n\rightarrow \mathbb{R}^N$ be an affine function independent of  $t$. Then, we have
\begin{equation*}
  \Xint-_{Q_\rho^{(\lambda)}(z_0)}|u-l_{z_0,\rho}^{(\lambda)}|^p dz\leq c(n,p) \Xint-_{Q_\rho^{(\lambda)}(z_0)}|u-l|^p dz.
\end{equation*}
\end{Lemma}
\subsection{$A-$caloric approximation}
A strongly elliptic bilinear form $A$ on $\mathbb{R}^{Nn}$ means that
\begin{equation*}
\nu|\widetilde{F}|^2\leq A(\widetilde{F},\widetilde{F}),\ \ \ \ \ \ A(F,\widetilde{F})\leq L|F||\widetilde{F}|,
\end{equation*}
for all $F,\widetilde{F}\in \mathbb{R}^{Nn}$ with ellipticity constant $\nu>0$ and upper bound $L>0$. We shall say that a function $h\in L^p(\Lambda_\rho(t_0);W^{1,p}(B_\rho(x_0);\mathbb{R}^N))$ is $A-caloric$ on $Q_\rho(z_0)$ if it satisfies
\begin{equation*}
\int_{Q_\rho(z_0)} h\cdot\varphi_t-A(Dh,D\varphi)dz=0,
\end{equation*}
for all $\varphi\in C_{0}^{\infty}(Q_\rho(z_0);\mathbb{R}^N)$.

In order to obtain the decay estimate (\ref{5.10}), we introduce the following $A-$caloric approximation lemma (cf. \cite{D18}).
\begin{Lemma}\label{le2.3}
 There exists a positive function $\delta_0=\delta_0(n,p,\nu,L,\varepsilon)\in (0,1]$
 with the following property, for each $\gamma\in (0,1]$, and each bilinear form $A$ in $\mathbb{R}^{Nn}$ with ellipticity constant $\nu$ and upper bound $L,\varepsilon$ is a positive number, whenever $u\in L^p(\Lambda_\rho(t_0);W^{1,p}(B_\rho(x_0);\mathbb{R}^N))$ satisfying
 \begin{equation*}
 \Xint-_{Q_\rho{(z_0)}}\left[\left(\left|\frac{u}{\rho}\right|^2+|Du|^2\right)\right]dz+\gamma^{p-2}\Xint-_{Q_\rho{(z_0)}}
 \left[\left|\frac{u}{\rho}\right|^p+|Du|^p\right]dz\leq 1,
 \end{equation*}
 is approximately $A-caloric$, in the sense that for each some $\delta\in(0,\delta_0]$ there holds
 \begin{equation*}
 \left|\Xint-_{Q_\rho{(z_0)}}(u\cdot\varphi_t-A(Du,D\varphi))dz\right|\leq\delta \sup_{Q_\rho(z_0)}|D\varphi|,
 \end{equation*}
 for all $\varphi\in C_{0}^{\infty}(Q_\rho(z_0);\mathbb{R}^N)$. Then, there exists an $A-$caloric function $h$ such that
 \begin{equation*}
 \Xint-_{Q_{\rho/2}{(z_0)}}\left[\left|\frac{h}{\rho/2}\right|^2+|Dh|^2\right]dz+\gamma^{p-2}\Xint-_{Q_{\rho/2}{(z_0)}}
 \left[\left|\frac{h}{\rho/2}\right|^p+|Dh|^p\right]dz\leq 2^{n+3+2p}
 \end{equation*}
   and
 \begin{equation*}
 \Xint-_{Q_{\rho/2}{(z_0)}}\left[\left|\frac{u-h}{\rho/2}\right|^2+\gamma^{p-2}\left|\frac{u-h}{\rho/2}\right|^p\right]dz\leq \varepsilon.
 \end{equation*}
\end{Lemma}
\section{Caccioppoli type inequality}\label{se3}
In this section, we propose to derive  a Caccioppoli type inequality under  the conditions (\ref{2.01})-(\ref{2.03}), (\ref{2.5}). Such result provide the smallness
condition in the iteration process of decay estimate in Section \ref{se5}.
\begin{Lemma}\label{le3.1}(Caccioppoli type inequality)
Let $u\in C^{0}(-T,0;L^2(\Omega;\mathbb{R}^N))\cap L^p(-T,0;W^{1,p}(\Omega;\mathbb{R}^N))$ be a weak solution to (\ref{1.1}) under the assumption
(\ref{2.01})-(\ref{2.03}), (\ref{2.5}), $Q_{\rho}^{(\lambda)}\subset Q_T$ is a scaled parabolic cylinder with reference point $z_0=(x_0,t_0)$ and $0<\rho\leq1$ suitable
small, scaling factor $\lambda\geq 1$ and affine function $l:\mathbb{R}^n\rightarrow \mathbb{R}^N$ such that $\lambda\leq 1+|Dl|$. Then there holds
\begin{align}\label{3.1}
&\Xint-_{Q_{\rho/2}^{(\lambda)}(z_0)}\left[\frac{|Du-Dl|^2}{(1+|Dl|)^2}+\frac{|Du-Dl|^p}{(1+|Dl|)^p}\right]dz\nonumber\\
\leq& c\left[\Xint-_{Q_{\rho}^{(\lambda)}(z_0)}\left[\frac{|u-l|^2}{\rho^2(1+|Dl|)^2}
+\frac{|u-l|^p}{\rho^p(1+|Dl|)^p}\right]dz\right.\nonumber\\
&\left.+\omega\left(\Xint-_{Q_\rho^{(\lambda)}(z_0)}|u-l(z_0)|^2 dz\right)+\omega(\rho^2)
+\rho\left(1+\Xint-_{Q^{\lambda}_{\rho}(z_0)}|u|^pdz\right)\right]
\end{align}
with $c=c(n,p,\nu,L)$.
\end{Lemma}
\begin{proof} Let $0<\rho/2\leq r<\sigma\leq\rho<1$, let $\eta(x)\in C_{0}^{1}(B_\sigma(x_0);\mathbb{R})$ be a cut-off function, $i.e$.
$0\leq \eta\leq1$,  $\eta=1$ in  $B_{r}(x_{0})$ and $|\nabla \eta|\leq\frac{c'}{\sigma-r}$ with $c'$ is a positive constant independent of $\sigma$ and
$r$. Moreover, we choose $\xi\in C_0^{1}(\Lambda_{\sigma}^{(\lambda)}(t_0))$ be a cut-off function in time, such that, with $0<\delta_1<r$ being arbitrary
\begin{equation*}
  \left\{
    \begin{array}{ll}
   \xi=1 & t\in (t_0-\lambda^{2-p}r^2,t_0-\delta_1),\\
   \xi=0 & t\in (-\infty,t_0-\lambda^{2-p}\sigma^2)\cup(t_0,0], \\
    0\leq\xi\leq1 & t\in \Lambda_{\sigma}^{(\lambda)}(t_0),\\
    |\xi_t|\leq\frac{2\lambda^{p-2}}{|\sigma-r|^2} & t\in (t_0-\lambda^{2-p}r^2,t_0].
    \end{array}
  \right.
  \end{equation*}
For simplicity, in what follows, we will omit the reference point $z_0$, and denote $Q_{\rho}^{(\lambda)}(z_0)$,$B_{\rho}^{(\lambda)}(x_0),\Lambda_{\rho}^{(\lambda)}(t_0)$ as $Q_{\rho}^{(\lambda)},B_\rho^{(\lambda)},$  $\Lambda_{\rho}^{(\lambda)}$, respectively.  Let $\varphi=\eta^{q_0+1}\xi^2(u-l)$ with
\begin{equation*}
q_0\geq  \max\left\{p-1,\frac{np(p-2)(n+2)}{2p[(n+2)p-2n]}-1\right\}
\end{equation*}
as a test function in the weak formulation $(\ref{2.1})$, which implies that
\begin{align}\label{3.3}
\int_{Q_\sigma^{(\lambda)}} a(z,u,Du)\cdot D(u-l)\xi^2\eta^{q_0+1} dz =&-(q_0+1)\int_{Q_\sigma^{(\lambda)}}a(z,u,Du)\cdot \xi^2\eta^{q_0}\nabla\eta\otimes(u-l)dz\nonumber\\
 & +\int_{Q_\sigma^{(\lambda)}}u\cdot\partial_t\varphi dz+\int_{Q_\sigma^{(\lambda)}}b(z,u,Du)\cdot\varphi dz.
 \end{align}
Observe  that
 \begin{align}\label{3.4}
 \int_{Q_\sigma^{(\lambda)}} a(z,u,Dl)\cdot D\varphi dz
 =&\int_{Q_\sigma^{(\lambda)}} a(z,u,Dl)\cdot D(u-l)\xi^2\eta^{q_0+1}dz\nonumber\\
 &+(q_0+1)\int_{Q_\sigma^{(\lambda)}}a(z,u,Dl)\cdot\xi^2\eta^{q_0}\nabla\eta\otimes(u-l) dz,
 \end{align}
 and
 \begin{equation}\label{3.5}
 \int_{Q_\sigma^{(\lambda)}} a(z_0,l(z_0),Dl)\cdot D\varphi dz=0.
 \end{equation}
Thus, inserting (\ref{3.4})-(\ref{3.5}) into (\ref{3.3}) and note that $l(z)=l(x)$, we arrive at
\begin{align}\label{3.6}
 \int_{Q_\sigma^{(\lambda)}}&[a(z,u,Du)-a(z,u,Dl)]\cdot D(u-l)\xi^2\eta^{q_0+1} dz\nonumber\\
 =&-(q_0+1)\int_{Q_\sigma^{(\lambda)}}[a(z,u,Du)-a(z,u,Dl)]\cdot\xi^2\eta^{q_0}\nabla\eta\otimes(u-l)dz\nonumber\\
 & -\int_{Q_\sigma^{(\lambda)}}[a(z,u,Dl)-a(z_0,u,Dl)]\cdot D\varphi dz\nonumber\\
 & -\int_{Q_\sigma^{(\lambda)}}[a(z_0,u,Dl)-a(z_0,l(z_0),Dl)]\cdot D\varphi dz\nonumber\\
 & +\int_{Q_\sigma^{(\lambda)}} (u-l)\cdot\partial_t\varphi dz
 +\int_{Q_\sigma^{(\lambda)}} b(z,u,Du)\cdot\varphi dz\nonumber\\
 :=& I+II+III+IV+V.
\end{align}
Firstly, we focus our attention on  estimating the term in the left side of (\ref{3.6}). Appealing to  (\ref{2.02}) and Lemma \ref{le2.1}, we infer that
\begin{align}\label{3.7}
 \int_{Q_\sigma^{(\lambda)}}&[a(z,u,Du)-a(z,u,Dl)]\cdot D(u-l)\xi^2\eta^{q_0+1} dz\nonumber\\
 &= \int_{Q_\sigma^{(\lambda)}}\int_{0}^{1}(\partial_F a(z,u,Dl+s(Du-Dl))(Du-Dl),Du-Dl)\xi^2\eta^{q_0+1} dsdz\nonumber\\
 & \geq\int_{Q_\sigma^{(\lambda)}}\int_{0}^{1}\nu(1+|Dl+s(Du-Dl)|)^{p-2}|Du-Dl|^2\xi^2\eta^{q_0+1} dsdz\nonumber\\
 & \geq\frac{\nu}{c(p)}\int_{Q_\sigma^{(\lambda)}}[(1+|Dl|)^{p-2}|Du-Dl|^2+|Du-Dl|^p]\xi^2\eta^{q_0+1} dz
\end{align}
Now, we turn to estimate the terms $I-V$ in (\ref{3.6}). For the term $I$, we first note that, from (\ref{2.01}) there holds
\begin{equation*}
|a(z,u,Du)-a(z,u,Dl)|\leq L(1+|Du|+|Dl|)^{p-2}|Du-Dl|,
\end{equation*}
and hence
\begin{align}\label{3.8}
 I\leq&\varepsilon\int_{Q_\sigma^{(\lambda)}}(1+|Dl|)^{p-2}|Du-Dl|^2\xi^2\eta^{2q_0} dz+c(L,\varepsilon)\int_{Q_\sigma^{(\lambda)}}(1+|Dl|)^{p-2}\xi^2|\nabla\eta|^2|u-l|^2 dz\nonumber\\
&+\varepsilon \int_{Q_\sigma^{(\lambda)}}|Du-Dl|^p\xi^2\eta^{\frac{q_0p}{p-1}}dz
+c(p,L,\varepsilon)\int_{Q_\sigma^{(\lambda)}}\xi^2\left|\frac{u-l}{\sigma-r}\right|^p dz,
\end{align}
where $\varepsilon\in (0,1)$ will be specified in later, and in the previous inequality we have taken into account that  $|Du|\leq |Dl|+|Du-Dl|$.

Next, using (\ref{2.03}), we deduce that
\begin{align}\label{3.9}
 II+III\leq & L(1+|Dl|)^{p-1}\int_{Q_\sigma^{(\lambda)}}\left[\omega(d(z,z_0)^2)+\omega(|u-l(z_0)|^2)\right]|D\varphi|dz\nonumber\\
 \leq& c(L)(1+|Dl|)^{p-1}\int_{Q_\sigma^{(\lambda)}}\left[\omega(\rho^2)+\omega (|u-l(z_0)|^2)\right]\left[|Du-Dl|\eta^{q_0+1}\xi^2 +\left|\frac{u-l}{\sigma-r}\right|\eta^{q_0}\xi^2\right] dz\nonumber\\
 \leq & \varepsilon\int_{Q_\sigma^{(\lambda)}}|Du-Dl|^p\eta^{p(q_0+1)}\xi^2 dz+\int_{Q_\sigma^{(\lambda)}}\left|\frac{u-l}{\sigma-r}\right|^p dz\nonumber\\
 & +c(p,L,\varepsilon)(1+|Dl|)^p\int_{Q_\sigma^{(\lambda)}}[\omega(\rho^2)+\omega(|u-l(z_0)|^2)]dz.
\end{align}
For the term $IV$, note that $\lambda\leq1+|Dl|$, we have
\begin{equation}\label{3.10}
 IV\leq\lambda^{p-2}\int_{Q_\sigma^{(\lambda)}}\left|\frac{u-l}{\sigma-r}\right|^2 dz
 \leq(1+|Dl|)^{p-2}\int_{Q_\sigma^{(\lambda)}}\left|\frac{u-l}{\sigma-r}\right|^2 dz.
\end{equation}
Finally, we estimate the term $V$.  From $(H_6)$, we have
\begin{align}\label{3.12}
 V&\leq L\int_{Q_\sigma^{(\lambda)}}(1+|Du|)^{p-1}|u-l|\eta^{q_0+1}\xi^2 dz+\int_{Q_\sigma^{(\lambda)}}|u|^{q_1}|u-l|\eta^{q_0+1}\xi^2 dz\nonumber\\
 &:=V_1+V_2.
 \end{align}
By the Young's inequality, it is clearly that
\begin{align}\label{3.120}
V_1& \leq \varepsilon_1\int_{Q_\sigma^{(\lambda)}}(\sigma-r)^{\frac{p}{p-1}}(1+|Du|)^p\eta^{q_0+1}\xi^2dz
+c(\varepsilon_1,L)\int_{Q_\sigma^{(\lambda)}}\left|\frac{u-l}{\sigma-r}\right|^p\eta^{q_0+1}\xi^2dz\nonumber\\
&\leq 2^{p-1}\varepsilon_1\int_{Q_\sigma^{(\lambda)}}\left[(1+|Dl|)^p+|Du-Dl|^p\right](\sigma-r)^{\frac{p}{p-1}}
\eta^{q_0+1}\xi^2dz+c(\varepsilon_1,L)\int_{Q_\sigma^{(\lambda)}}
\left|\frac{u-l}{\sigma-r}\right|^p\eta^{q_0+1}\xi^2dz
\end{align}
with $\varepsilon_1\in (0,1)$ will be specified in later.

For the term $V_2$, first, we divide $B_{\sigma}$ into two parts: $B_{\sigma}:=B_{\sigma}^1\cup B_{\sigma}^2$ and
\begin{align*}
B_{\sigma}^1:= &\left\{x\in B_{\sigma}:\ 0\leq\eta(x)<(\sigma-r)^{\frac{n+2}{q_0+1}}\right\},\\
B_{\sigma}^2:= &\left\{x\in B_{\sigma}:\ (\sigma-r)^{\frac{n+2}{q_0+1}}\leq\eta(x)\leq1\right\}.
\end{align*}
Next, we take $k$, $q$, $p_1$, $p_0$, $\alpha$, $r$, $s$, $\alpha_0$ in (\ref{1.3}) as $0$, $\frac{(n+2)p}{n}$, $p$, $2$,  $q_0+2$, $1$, $q_0+1$,
$q_0+1-n(\frac12-\frac1p)$, respectively, and at this stage, we have $\theta=\frac{n}{n+2}$ in (\ref{1.3}). Moreover, by the definition of $q_0$, we can
see that
\begin{equation*}
(\alpha_0-1)\frac{(n+2)p}{2n}\geq q_0+1.
\end{equation*}
Therefore, by weighted Sobolev interpolation inequality (\ref{1.3}) and the H\"{o}lder's inequality, we are in a position to obtain
\begin{align*}
V_2\leq&\varepsilon_2(\sigma-r)^{\frac{p}{p-1}}\int_{Q_\sigma^{(\lambda)}}|u|^{\frac{q_1p}{p-1}}\eta^{q_0+1}\xi^2dz
+c(\varepsilon_2)\underbrace{\int_{Q_\sigma^{(\lambda)}} \left|\frac{u-l}{\sigma-r}\right|^p\eta^{q_0+1}\xi^2dz}_{V_{21}} \nonumber\\
\leq&\varepsilon_2(\sigma-r)^{\frac{p}{p-1}}\int_{Q_{\sigma}^{(\lambda)}}(1+|u|^{\frac{n+2}{n}p})\eta^{q_0+1}\xi^2dz+c(\varepsilon_2)V_{21}\nonumber\\
\leq&c(n,p)\varepsilon_2(\sigma-r)^{\frac{p}{p-1}}\int_{\Lambda_{\sigma}^{(\lambda)}}\left[\int_{B_{\sigma}^2}|Du|^p\eta^{q_0+1}dx
\left(\int_{B_{\sigma}^2}|u|^2\eta^{\alpha_0-1}dx\right)^{\frac{p}{n}}
+\left(\int_{B_{\sigma}^2}|u|^2\eta^{\alpha_0-1}dx\right)^{\frac{(n+2)p}{2n}}\right]\xi^2dt\nonumber\\
&+c(n,p)\varepsilon_2(\sigma-r)^{\frac{p}{p-1}}|Q_{\sigma}^{(\lambda)}|+c(\varepsilon_2)V_{21}
\end{align*}
\begin{align}\label{3.122}
\leq&c(n,p)\varepsilon_2(\sigma-r)^{\frac{p}{p-1}}\int_{Q_{\sigma}^{(\lambda)}}|Du|^p\eta^{q_0+1}\xi^2dz
+\varepsilon_2(\sigma-r)^{\frac{p}{p-1}}\underbrace{ c\int_{\Lambda_{\sigma}^{(\lambda)}}|B_{\sigma}|^{\frac{(n+2)p-2n}{2n}}
\int_{B_{\sigma}^2}|u|^{\frac{(n+2)p}{n}}\eta^{q_0+1}\xi^2dz}_{V_{22}}\nonumber\\
&+c\varepsilon_2(\sigma-r)^{\frac{p}{p-1}}|Q_{\sigma}^{(\lambda)}|+c(\varepsilon_2)V_{21},
\end{align}
where $\varepsilon_2\in (0,1)$ will be specified in later.

Now, we choose $\rho$ suitable small such that
\begin{equation*}
c|B_{\sigma}|^{\frac{(n+2)p-2n}{2n}}\leq\frac12.
\end{equation*}
Therefore, we have
\begin{align*}
V_{22} & \leq\frac12\int_{\Lambda_{\sigma}^{(\lambda)}}\int_{B_{\sigma}^2}|u|^{\frac{(n+2)p}{n}}\eta^{q_0+1}dz\nonumber\\
 &\leq\frac12 c(n,p)\int_{Q_\sigma^{(\lambda)}}|Du|^p\eta^{q_0+1}dz+\frac12
 \underbrace{ c\int_{\Lambda_{\sigma}^{(\lambda)}}|B_{\sigma}|^{\frac{(n+2)p-2n}{2n}}
\int_{B_{\sigma}^2}|u|^{\frac{(n+2)p}{n}}\eta^{q_0+1}\xi^2dz}_{V_{22}}.
\end{align*}
This implies that
\begin{equation}\label{3.212}
V_{22}\leq c(n,p)\int_{Q_\sigma^{(\lambda)}}|Du|^p\eta^{q_0+1}dz
\end{equation}
As a consequence, from (\ref{3.122})--(\ref{3.212}), it follows that
\begin{align}\label{3.15}
V_2&\leq c(n,p)\varepsilon_2(\sigma-r)^{\frac{p}{p-1}}\int_{Q^{(\lambda)}_{\sigma}}|Du|^p\eta^{q_0+1}\xi^2dz
+ c(n,p,\varepsilon_2)(\sigma-r)^{\frac{p}{p-1}}|Q_{\sigma}^{(\lambda)}|+c(\varepsilon_2)V_{21}\nonumber  \\
& \leq  c(n,p)\varepsilon_2(\sigma-r)^{\frac{p}{p-1}}\int_{Q^{(\lambda)}_{\sigma}}|Du-Dl|^p\eta^{q_0+1}\xi^2dz
+ c(n,p,\varepsilon_2)(\sigma-r)^{\frac{p}{p-1}}(1+|Dl|^p)|Q_{\sigma}^{(\lambda)}|+c(\varepsilon_2)V_{21}.
\end{align}
Inserting (\ref{3.7})-(\ref{3.120}) and (\ref{3.15}) into (\ref{3.6}), we conclude that
\begin{align*}
 &\frac{\nu}{c(p)}\int_{Q_\sigma^{(\lambda)}}\left[(1+|Dl|)^{p-2}|Du-Dl|^2+|Du-Dl|^p \right]\eta^{q_0+1}\xi^2 dz\nonumber\\
  \leq &\varepsilon\int_{Q_\sigma^{(\lambda)}}(1+|Dl|)^{p-2}|Du-Dl|^2\eta^{2q_0}\xi^2dz +c(\varepsilon,L)\int_{Q_\sigma^{(\lambda)}}\left|\frac{u-l}{\sigma-r}\right|^2(1+|Dl|)^{p-2}dz\nonumber\\
 & +\varepsilon\int_{Q_\sigma^{(\lambda)}}|Du-Dl|^p(\eta^{\frac{q_0p}{p-1}}+\eta^{p(q_0+1)})\xi^2dz
 +c(n,p,L,\varepsilon,\varepsilon_1)\int_{Q_\sigma^{(\lambda)}}\left|\frac{u-l}{\sigma-r}\right|^pdz\nonumber\\
 & +c(p,L,\varepsilon)(1+|Dl|)^{p}\int_{Q_\sigma^{(\lambda)}}[\omega(\rho^2)+\omega(|u-l(z_0)|^2)]dz\nonumber\\
 & +2^{p-1}\varepsilon_1\int_{Q_{\sigma}^{(\lambda)}}[(1+|Dl|)^p+|Du-Dl|^p](\sigma-r)^{\frac{p}{p-1}}\eta^{q_0+1}\xi^2dz\nonumber\\
& +c(n,p)\varepsilon_2(\sigma-r)^{\frac{p}{p-1}}\int_{Q_{\sigma}^{(\lambda)}}|D(u-l)|^p\eta^{q_0+1}\xi^2 dz
+c(n,p,\varepsilon_2)(1+|Dl|)^p|Q_{\sigma}|(\sigma-r)^{\frac{p}{p-1}}.
\end{align*}
Now, we choose $\varepsilon=\frac{\nu}{8c(p)}$ and $\varepsilon_1=2^{-(p+2)}\frac{\nu}{c(p)}$, $\varepsilon_2=\frac{\nu}{8c(n,p)c(p)}$, note that
$2q_0,\frac{pq_0}{p-1}\geq q_0+1$ and $\frac{p}{p-1}\in (1,2]$, moving the same terms into left side, and dividing by $(1+|Dl|)^p$ in
both side, taking mean values and Jensen's inequality
\begin{equation*}
\Xint-_{Q_\sigma^{(\lambda)}}\omega(|u-l(z_0)|^2)dz\leq\omega\left(\Xint-_{Q_\sigma^{(\lambda)}}|u-l(z_0)|^2 dz\right),
\end{equation*}
letting $\delta_1\rightarrow0$, then we have (\ref{3.1}).
\end{proof}
\section{Poincar\'{e} type inequality}\label{se4}
In this section, we aim at  establishing a Poincar\'{e} type inequality of weak solution to (\ref{1.1}) under the assumptions (\ref{2.01}), (\ref{2.03}),
(\ref{2.5}). We note that such inequality  plays a key role in the whole paper, that will be used in Sec. \ref{se5}, from which, we are able to show
that for every $z_0\in Q_T\setminus(\Sigma_1\cup\Sigma_2)$ and  suitable $0\leq\rho_0\leq1$, the assumption of Lemma \ref{le5.3} is valid.
\begin{Lemma}\label{le4.1}(Poincar\'{e} type inequality).
 Let $u\in L^p(-T,0;W^{1,p}(\Omega;\mathbb{R}^N))\cap C^{0}(-T,0;L^2(\Omega;\mathbb{R}^N))$ be a weak solution of (\ref{1.1}) in $Q_T$ under the assumption (\ref{2.01}), (\ref{2.03}), (\ref{2.5}), for any $Q_\rho{(z_0)}\subset Q_T$ be a parabolic cylinder with reference $z_0=(x_0,t_0)$ and $0<\rho\leq1$. Then, there holds
\begin{equation}\label{4.1}
\Xint-_{Q_\rho{(z_0)}}\left|\frac{u-(u)_{z_0,\rho}}{\rho}\right|^q dz\leq c\left(\Xint-_{Q_\rho{(z_0)}}(1+|Du|)^p dz\right)^q,
\end{equation}
for any $q\in[1,p]$, and
\begin{align}\label{4.2}
\Xint-_{Q_\rho{(z_0)}}&\left|\frac{u-(u)_{z_0,\rho}-(Du)_{z_0,\rho}{(x-x_0)}}{\rho}\right|^p dz\nonumber\\
 \leq &c\Xint-_{Q_\rho{(z_0)}}|Du-(Du)_{z_0,\rho}|^p dz\left[(1+|(Du)_{z_0,\rho}|)^p+\Xint-_{Q_\rho{(z_0)}}|Du-(Du)_{z_0,\rho}|^p dz\right]^{p-2}\nonumber\\
& +c(1+|(Du)_{z_0,\rho}|)^{p(p-1)}[\omega(\rho^2)]^p \left(\Xint-_{Q_\rho{(z_0)}}(1+|Du|)^p dz\right)^{2p}\nonumber\\
& +c\rho^p\left[(1+|(Du)_{z_0,\rho}|)^{p^2}+\left(\Xint-_{Q_\rho{(z_0)}}|Du-(Du)_{z_0,\rho}|^p dz\right)^p\right],
\end{align}
where $c=c(n,N,p,L)$.
\end{Lemma}
\begin{proof} For simplicity, we may also omit the reference point $z_0$ of $Q_\rho{(z_0)}$, $B_\rho{(z_0)}$, $\Lambda_\rho{(z_0)}$, instead by $Q_\rho$, $B_\rho$, and $\Lambda_\rho$, respectively, if there is no  danger of any  confusion. Let $\eta\in C_{0}^{\infty}(B_\rho)$ be a nonnegative weight function satisfying
\begin{equation*}
  \left\{
    \begin{array}{ll}
 \eta\geq0,\\
   \Xint-_{B_\rho}\eta dx=1, \\
    \|\eta\|_{L^{\infty}}+\rho\|\nabla\eta\|_{L^{\infty}}\leq c_\eta,
    \end{array}
  \right.
\end{equation*}
where $c_\eta=c_{\eta}(n)$, define
\begin{equation*}
(u)_\eta(t)=\Xint-_{B_\rho} u(x,t)\eta dx,
\end{equation*}
as a weighted mean of $u(x,t)$ on $B_\rho$ for a.e. $t\in(-T,0)$. To begin with, we shall show the following argument for a.e. $t,\tau\in\Lambda_\rho$:
\begin{equation}\label{4.7}
|(u)_\eta(t)-(u)_\eta(\tau)|\leq c\rho\Xint-_{Q_\rho}(1+|Du|)^{p-1} dz+c\rho^2\Xint-_{Q_\rho}(1+|Du|)^{p} dz,
\end{equation}
and
\begin{align}\label{4.8}
|(u)_\eta(t)-(u)_\eta(\tau)|\leq& c\rho\Xint-_{Q_\rho}\left[(1+|(Du)_\rho|)^{p-2}|Du-(Du)_\rho|+|Du-(Du)_\rho|^{p-1}\right]dz\nonumber\\
& +c\rho (1+|(Du)_\rho|)^{p-1}\left[\omega \left(\Xint-_{Q_\rho}|u-(u)_\rho|^2dz\right)+\omega_(\rho^2)\right]\nonumber\\
& +c\rho^2\left[(1+|(Du)_\rho|)^{p}+\Xint-_{Q_\rho}|Du-(Du)_\rho|^{p-1} dz+\Xint-_{Q_\rho}|Du-(Du)_\rho|^{p} dz\right],
\end{align}
where $c=c(n,N,p,L)$.

Now, we concentrate our attention on the proof of (\ref{4.7})-(\ref{4.8}),  without loss of generality, we may assume $t>\tau$, let $\xi_\theta(s)\in C_{0}^{\infty}((\tau,t))$ be a cut-off function, defined by
\begin{align*}
\xi_\theta=
\begin{cases}
 1  &s\in [\tau+\theta,t-\theta],  \\
 \frac{s-\tau}{\theta}&s\in(\tau,\tau+\theta),\\
  -\frac{s-t}{\theta}&s\in(t-\theta,t),
  \end{cases}
\end{align*}
with $\theta\in (0,(t-\tau)/2)$. We now choose $\varphi_\theta:\mathbb{R}^{n+2}\rightarrow \mathbb{R}^N$ be a test function in the weak formulation (\ref{2.1}) with $(\varphi_\theta)_i=\eta\xi_\theta$ and $(\varphi_\theta)_j=0$ for $j\neq i$ and $i,j\in \{1,\cdots,N\}$, which implies
\begin{equation}\label{4.9}
\int_{\tau}^{t}\Xint-_{B_\rho}u_i\eta\partial_s\xi_\theta dxds=\int_{\tau}^{t}\Xint-_{B_\rho}a_i(z,u,Du)\cdot\nabla\eta\xi_\theta-b_i(z,u,Du)\cdot\eta\xi_\theta dxds.
\end{equation}
Taking into account the  Steklov arguments  and the definition of $(u)_\eta(t)$, we first deduce that
\begin{align*}
\int_{\tau}^{t}\Xint-_{B_\rho}u_i\eta\partial_s\xi_\theta dxds&=\int_{\tau}^{t}(u_i)_\eta\partial_s\xi_\theta ds\nonumber\\
& =\frac{1}{\theta}\int_{\tau}^{\tau+\theta}(u_i)_\eta ds-\frac{1}{\theta}\int_{t-\theta}^{t}(u_i)_\eta ds\rightarrow(u_i)_\eta(\tau)-(u_i)_\eta(t),
\end{align*}
as $\theta\rightarrow0$.

Next, letting $\theta\rightarrow0$ in the right side of (\ref{4.9}), we arrive at
\begin{equation}\label{4.10}
(u_i)_\eta(\tau)-(u_i)_\eta(t)=\int_{\tau}^{t}\Xint-_{B_\rho}a_i((x,s),u,Du)\cdot\nabla\eta dxds-\int_{\tau}^{t}\Xint-_{B_\rho}b_i((x,s),u,Du)\cdot\eta dxds.
\end{equation}
In virtue of  (\ref{2.5}), (\ref{4.10}), and note that $t,\tau\in \Lambda_\rho$,   we infer that
\begin{align}\label{4.91}
|(u_i)_\eta(\tau)-(u_i)_\eta(t)|&\leq L\|\nabla\eta\|_{L^{\infty}}\rho^2\Xint-_{Q_\rho}(1+|Du|)^{p-1}dz+\rho^2\|\eta\|_{L^{\infty}}\Xint-_{Q_\rho}(1+|Du|)^{p-1} dz+\rho^2\|\eta\|_{L^{\infty}}\Xint-_{Q_\rho}|u|^{q_1}dz\nonumber\\
& \leq c(n,L)\rho\Xint-_{Q_\rho}(1+|Du|)^{p-1} dz+\rho^2\|\eta\|_{L^{\infty}}\Xint-_{Q_\rho}|u|^{q_1}dz.
\end{align}
Now we focus our attention on estimating the term $W$. Employing  interpolation inequality (G-N-S inequality), it holds that
\begin{align}\label{4.92}
\Xint-_{Q_\rho}|u|^{q_1}dz\leq & \Xint-_{Q_\rho} 1+|u|^{\frac{n+2}{n}p} dz\nonumber\\
  \leq& 1+c\frac{1}{|Q_{\rho}|}\int_{\Lambda_{\rho}}\|u\|_{L^2(B_{\rho})}^{\frac{(n+2)p\theta_1}{n}}(\|u\|_{L^p(B_{\rho})}^{\frac{(n+2)p(1-\theta_1)}{n}}
  +\|Du\|_{L^p(B_{\rho})}^{\frac{(n+2)p(1-\theta_1)}{n}})\nonumber\\
 \leq& 1+c\sup_{t\in \Lambda_{\rho}}\|u\|_{L^2(B_{\rho})}^{\frac{(n+2)p\theta_1}{n}}\Xint-_{Q_\rho}|Du|^pdz+\frac{1}{|Q_{\rho}|}
 \underbrace{c\int_{\Lambda_{\rho}}\|u\|_{L^2(B_{\rho})}^{\frac{(n+2)p\theta_1}{n}}|B_{\rho}|^{\frac{2}{n+2}}
 \left(\int_{B_{\rho}}|u|^{\frac{n+2}{n}p}dx\right)^{\frac{n}{n+2}}dt}_{W_1},
\end{align}
where  in the last inequality we have taken into account that
\begin{equation}\label{3.123}
  \frac{n+2}{n}p(1-\theta_1)=p,\quad \theta_1=\frac{2}{n+2}.
\end{equation}
It is clearly that the term $W_1$ can be split as
 \begin{align}\label{4.93}
    W_1=&\int_{J_1}(\cdots)dt+ \int_{J_2}(\cdots)dt\nonumber \\
    :=&W_{11} +W_{12},
 \end{align}
where $\Lambda_{\rho}=J_1\cup J_2$, and
\begin{align*}
  J_1:= & \left\{t\in \Lambda_{\rho}:\left(c\|u\|_{L^2(B_{\rho})}^{\frac{(n+2)p\theta_1}{n}}\right)^{\frac{n+2}{2}}
  |B_{\rho}|\geq (\frac{1}{2})^{\frac{n+2}{2}}\int_{B_{\rho}}|u|^{\frac{n+2}{n}p}dx\right\}, \\
  J_2:= & \left\{t\in \Lambda_{\rho}:\left(c\|u\|_{L^2(B_{\rho})}^{\frac{(n+2)p\theta_1}{n}}\right)^{\frac{n+2}{2}}
  |B_{\rho}|<(\frac{1}{2})^{\frac{n+2}{2}}\int_{B_{\rho}}|u|^{\frac{n+2}{n}p}dx\right\}.
\end{align*}
Thus, we are in a position to obtain
\begin{equation}\label{4.94}
 W_{11}\leq c|Q_{\rho}|,
\end{equation}
and by iteratively estimating, we have
\begin{align}\label{4.95}
  W_{12}\leq& \frac12\int_{J_2}\int_{B_{\rho}}|u|^{\frac{n+2}{n}p} dz\nonumber \\
   \leq& \frac12 c\int_{J_2}\|u\|_{L^2(B_{\rho})}^{\frac{(n+2)p\theta_1}{n}}(\|u\|_{L^p(B_{\rho})}^{\frac{(n+2)p(1-\theta_1)}{n}}
  +\|Du\|_{L^p(B_{\rho})}^{\frac{(n+2)p(1-\theta_1)}{n}})dt\nonumber \\
 \leq& \frac12 c\sup_{t\in \Lambda_{\rho}} \|u\|_{L^2(B_{\rho})}^{\frac{(n+2)p\theta_1}{n}}\int_{Q_{\rho}}|Du|^pdz
 +\frac12 c \int_{J_2}\|u\|_{L^2(B_{\rho})}^{\frac{(n+2)p\theta_1}{n}}|B_{\rho}|^{\frac{2}{n+2}}
 \left(\int_{B_{\rho}}|u|^{\frac{n+2}{n}p}dx\right)^{\frac{n}{n+2}}dt\nonumber \\
 \leq& \frac12 c\sup_{t\in \Lambda_{\rho}} \|u\|_{L^2(B_{\rho})}^{\frac{(n+2)p\theta_1}{n}}\int_{Q_{\rho}}|Du|^pdz
 +\frac14\int_{J_2}\int_{B_{\rho}}|u|^{\frac{n+2}{n}p}dz\nonumber \\
 \vdots&\nonumber \\
 \leq&c\sup_{t\in \Lambda_{\rho}} \|u\|_{L^2(B_{\rho})}^{\frac{(n+2)p\theta_1}{n}}\int_{Q_{\rho}}|Du|^pdz.
\end{align}
Plugging (\ref{4.93})-(\ref{4.95}) into (\ref{4.92}), we conclude that
\begin{equation}\label{4.96}
\Xint-_{Q_{\rho}}|u|^{q_1}dz\leq c(n,p)\left(1+\Xint-_{Q_{\rho}}|Du|^pdz\right).
\end{equation}
Now, combining (\ref{4.96}) and (\ref{4.91}), and summing up over $i=1,\cdots,N$,  then we have  (\ref{4.7}). Hence, it remains to prove  (\ref{4.8}).

Observing that
\begin{equation*}
\int_{\tau}^{t}\Xint-_{B_\rho}a_i((x_0,t),(u)_\rho,(Du)_\rho)\cdot\nabla\eta dxds=0,
\end{equation*}
Making use of (\ref{4.10}), then we infer that
\begin{align}\label{4.11}
|(u_i)_\eta(t)-(u_i)_\eta(\tau)|\leq& \rho^2\|\nabla\eta\|_{L^{\infty}} \left[\Xint-_{Q_\rho}|a(z,u,Du)-a(z,u,(Du)_\rho)|dz\right. \nonumber\\
& +\Xint-_{Q_\rho}|a(z,u,(Du)_{\rho})-a(z,(u)_\rho,(Du)_\rho)|dz\nonumber\\
&\left. +\Xint-_{Q_\rho}|a(z,(u)_\rho,(Du)_\rho)-a((x_0,t),(u)_\rho,(Du)_\rho)|dz\right]+\rho^2\|\eta\|_{L^{\infty}}\Xint-_{Q_\rho}|b(z,u,Du)|dz\nonumber\\
 :=&\rho^2\|\nabla\eta\|_{L^{\infty}}(K_1+K_2+K_3)+\rho^2\|\eta\|_{L^{\infty}} K_4.
\end{align}
Applying (\ref{2.01}) and Lemma \ref{le2.1}, for the term $K_1$, we have
\begin{align}\label{4.12}
 K_1=&\Xint-_{Q_\rho}\left|\int_{0}^{1}\partial_F a(z,u,(Du)_\rho+s(Du-(Du)_\rho))\cdot(Du-(Du)_\rho)ds\right|dz\nonumber\\
 \leq& L\Xint-_{Q_\rho}\int_{0}^{1}(1+|(Du)_\rho+s(Du-(Du)_\rho)|)^{p-2}ds|Du-(Du)_\rho|dz\nonumber\\
 \leq& c(p)L\Xint-_{Q_\rho}(1+|(Du)_\rho|)^{p-2}|Du-(Du)_\rho|+|Du-(Du)_\rho|^{p-1}dz.
\end{align}
In addition, making use of (\ref{2.03}) and Jensen's inequality, the term $K_2$ and $K_3$ can be estimated as
\begin{align}\label{4.13}
  K_2+K_3&=L(1+|(Du)_{\rho}|)^{p-1}\left[\Xint-_{Q_\rho}\omega(|u-(u)_\rho|^2)dz+\omega(\rho^2)\right]\nonumber\\
 & \leq L(1+|(Du)_{\rho}|)^{p-1}\left[\omega\left(\Xint-_{Q_\rho}|u-(u)_\rho|^2dz\right)+\omega(\rho^2)\right].
\end{align}
For the term $K_4$, in view of (\ref{4.10})-(\ref{4.91}) and  (\ref{4.96}), we have
\begin{equation}\label{4.15}
K_4\leq c(n,p,L)\left[(1+|(Du)_\rho|)^{p}+\Xint-_{Q_\rho}|Du-(Du)_\rho|^{p-1}dz+\Xint-_{Q_\rho}|Du-(Du)_\rho|^{p}dz\right].
\end{equation}
Inserting (\ref{4.12})-(\ref{4.15}) into (\ref{4.11}), summing up over $i=1,\cdots,N$, whence  (\ref{4.8}).

Now, we turn to prove  (\ref{4.1})-(\ref{4.2}).  First, appealing to  (\ref{4.7}), Poincar\'{e}'s inequality with weighed function, H\"{o}lder's inequality, we infer that
\begin{align*}
\Xint-_{Q_\rho}\left|\frac{u-(u)_\rho}{\rho}\right|^q dz
&\leq 3^{q-1}\left[\Xint-_{Q_\rho}\left|\frac{u-(u)_{\eta}(t)}{\rho}\right|^qdz+\rho^{-q}\left|\Xint-_{\Lambda_{\rho}}(u)_{\eta}(t)dt
-\Xint-_{\Lambda_{\rho}}(u)_{\eta}(\tau)d\tau\right|^q+\rho^{-q}\left|\Xint-_{Q_{\rho}}u-(u)_{\eta} d\tilde{z}\right|^q\right]\nonumber\\
&\leq 3^{p-1}\left[2c(n,q)\Xint-_{Q_\rho}|Du|^q dz+\rho^{-q} \sup_{t,\tau\in \Lambda_\rho}|(u)_\eta(t)-(u)_\eta(\tau)|^q\right]\nonumber\\
& \leq c\left[\left(\Xint-_{Q_\rho}|Du|^p dz\right)^{\frac{q}{p}}+\left(\Xint-_{Q_\rho}(1+|Du|)^{p-1} dz\right)^q
+\rho^q\left(\Xint-_{Q_\rho}(1+|Du|)^{p} dz\right)^q\right]\nonumber\\
&\leq c\left(\Xint-_{Q_\rho}(1+|Du|)^{p} dz\right)^q,
\end{align*}
where $\tilde{z}=(x,\tau)$ and $c=c(n,N,p,L)$. Thus, we have (\ref{4.1}).

Next, by (\ref{4.8}),  Poincar\'{e} and H\"{o}lder's inequality, we obtain
\begin{align}\label{4.012}
\Xint-_{Q_\rho}\left|\frac{u-(u)_\rho-(Du)_\rho x}{\rho}\right|^p dz
\leq& 3^{p-1}\left[\Xint-_{Q_\rho}\left|\frac{u-(u)_\eta(t)-(Du)_\rho x}{\rho}\right|^p dz+\rho^{-p}\left|\Xint-_{\Lambda_\rho}(u)_\eta(t)dt-\Xint-_{\Lambda_\rho}(u)_\eta(\tau)d\tau\right|^p\right.\nonumber\\
& \left.+\rho^{-p}\left|\Xint-_{\Lambda_\rho}(u)_\eta(\tau)d\tau-(u)_\rho\right|^p \right]\nonumber\\
\leq& 3^{p-1}\left[c(n,p)\Xint-_{Q_\rho}|Du-(Du)_\rho|^p dz+\rho^{-p} \sup_{t,\tau\in \Lambda_\rho}|(u)_\eta(t)-(u)_\eta(\tau)|^p\right.\nonumber\\
&\left. +\rho^{-p}\left|\Xint-_{Q_\rho}u-(u)_\eta(\tau)-(Du)_\rho x d\tilde{z}\right|^p\right]\nonumber\\
\leq &c(n,p)\left[\Xint-_{Q_\rho}|Du-(Du)_\rho|^p dz+\rho^{-p} \sup_{t,\tau\in \Lambda_\rho}|(u)_\eta(t)-(u)_\eta(\tau)|^p\right]\nonumber\\
\leq& c\Xint-_{Q_\rho}|Du-(Du)_\rho|^p dz \left[(1+|(Du)_\rho|)^p+\Xint-_{Q_\rho}|Du-(Du)_\rho|^p dz\right]^{p-2}\nonumber\\
& +c(1+|(Du)_\rho|)^{p(p-1)}\left[\omega\left(\Xint-_{Q_\rho}|u-(u)_\rho|^2 dz\right)+\omega(\rho^2)\right]^p\nonumber\\
& +c\rho^p\left[(1+|(Du)_\rho|)^{p^2}+\left(\Xint-_{Q_\rho}|Du-(Du)_{\rho}|^p dz\right)^p\right]
\end{align}
with $c=c(n,N,p,L)$, where in the second inequality, we have used the Poincar\'{e}'s inequality for a.e. $t\in \Lambda_{\rho}$ and the  fact
 \begin{equation*}
 \Xint-_{Q_\rho}(Du)_\rho xdz=0.
\end{equation*}
Taking into account the concavity of $\omega(\cdot)$ and (\ref{4.1}) for $q=2$ implies
\begin{align}\label{4.013}
\omega\left(\Xint-_{Q_\rho}|u-(u)_\rho|^2 dz\right)+\omega(\rho^2)&\leq\omega\left(c\rho^2\left(\Xint-_{Q_\rho}(1+|Du|)^p dz\right)^2\right)\nonumber\\
&\leq c\left(\Xint-_{Q_\rho}(1+|Du|)^p dz\right)^2 \omega(\rho^2).
\end{align}
Thus, Combining  (\ref{4.012}) and (\ref{4.013}), we are in a position to obtain
\begin{align*}
\Xint-_{Q_\rho}\left|\frac{u-(u)_\rho-(Du)_\rho x}{\rho}\right|^p dz\leq& c\Xint-_{Q_\rho}|Du-(Du)_\rho|^p dz \left[(1+|(Du)_\rho|)^p+\Xint-_{Q_\rho}|Du-(Du)_\rho|^p dz\right]^{p-2}\nonumber\\
& +c(1+|(Du)_\rho|)^{p(p-1)}\left[\omega(\rho^2)\right]^p\left(\Xint-_{Q_\rho}(1+|Du|)^p dz\right)^{2p}\nonumber\\
& +c\rho^p\left[(1+|(Du)_\rho|)^{p^2}+\left(\Xint-_{Q_\rho}|Du-(Du)_\rho|^p dz\right)^p\right],
\end{align*}
whence (\ref{4.2}).
\end{proof}
\section {Partial regularity of $u$}\label{se5}
According to Lemma \ref{le3.1}, now, we define some excess functionals. For reference point $z_0=(x_0,t_0)\in Q_T$,
$u\in L^p(-T,0;W^{1,p}(\Omega;\mathbb{R}^N))$, affine function $l:\mathbb{R}^n\rightarrow \mathbb{R}^N$, and $l(z)=l(x)$, in what follows, we  denote
\begin{equation*}
\text{\emph{first\  order\ excess}:}\quad \Phi_\lambda(\rho)\equiv\Phi_\lambda(u;z_0,\rho,l)
:=\Xint-_{Q_{\rho}^{(\lambda)}(z_0)}\left[\frac{|u-l|^2}{\rho^2(1+|Dl|)^2}+\frac{|u-l|^p}{\rho^p(1+|Dl|)^p}\right]dz
\end{equation*}
\begin{equation*}
\text{\emph{zero\  order\ excess}:}\quad\Psi_\lambda(\rho)\equiv\Psi_\lambda(u;z_0,\rho,l(z_0)):=\Xint-_{Q_{\rho}^{(\lambda)}(z_0)}|u-l(z_0)|^2dz
\end{equation*}
and \emph{hybrid excess functional}:
\begin{equation*}
E_\lambda(\rho)\equiv E_\lambda(u;z_0,\rho,l):=\Phi_\lambda(\rho)+\omega\left(\Psi_\lambda(z_0,\rho,l(z_0))\right)+\omega(\rho^2)+\rho.
\end{equation*}
\subsection{Linearization}
The following lemma is a prerequisite for applying the $A-$caloric approximation technique.
\begin{Lemma}\label{le5.1}
Let $u\in L^p(-T,0;W^{1,p}(\Omega;\mathbb{R}^N))\cap C^{0}(-T,0;L^2(\Omega;\mathbb{R}^N))$ is a weak solution to (\ref{1.1}) in $Q_T$ under the assumption (\ref{2.01})-(\ref{2.5}) and $Q_{\rho}^{(\lambda)}(z_0)\subset Q_T$ is a parabolic cylinder with reference point $z_0=(x_0,t_0)$, $0<\rho\leq1$ and scaling factor $\lambda\geq1$. Let $l:\mathbb{R}^n\longrightarrow\mathbb{R}^N$ be any affine function. Then, there holds
\begin{align}\label{5.1}
&\left|\Xint-_{Q_{\rho/2}^{(\lambda)}(z_0)}\left[(u-l)\cdot\varphi_t-(\partial _F a(z_0,l(z_0),Dl)(Du-Dl),D\varphi)\right]dz\right|\nonumber\\
& \leq c(1+|Dl|^{p-1})\left[E_\lambda(\rho)+\mu(\sqrt{E_\lambda(\rho)})\right]^{\frac12}\sqrt{E_\lambda(\rho)}
 \sup_{z\in Q_{\rho/2}^{(\lambda)}(z_0)}|D\varphi|,
\end{align}
for all $\varphi\in C_{0}^{\infty}(Q_{\rho/2}^{(\lambda)}(z_0);\mathbb{R}^N)$ with $c=c(n,L,p,\nu)$.
\end{Lemma}
\begin{proof} Without loss of generality, we may assume $\sup_{Q_{\rho/2}^{(\lambda)}(z_0)}|D\varphi|\leq1$ and we also denote $Q^{\lambda}_\rho$, $B_\rho$, $\Lambda_\rho$ instead of $Q^{\lambda}_\rho(z_0)$, $B_\rho(x_0)$, $\Lambda_\rho(t_0)$, respectively,  if there is no danger of any confusion. Note that
\begin{equation*}
\Xint-_{Q_{\rho/2}^{(\lambda)}} l\cdot \varphi_t dz=0\ \ \ \text{and}\ \ \  \Xint-_{Q_{\rho/2}^{(\lambda)}} a(z_0,l(z_0),Dl)\cdot D\varphi dz=0,
\end{equation*}
then, from weak formulation (\ref{2.1}), we deduce that
\begin{align}\label{5.2}
 \Xint-_{Q_{\rho/2}^{(\lambda)}}&[(u-l)\cdot \varphi_t-\partial _F a(z_0,l(z_0),Dl)(Du-Dl,D\varphi)]dz\nonumber\\
 = &\Xint-_{Q_{\rho/2}^{(\lambda)}}[( a(z_0,l(z_0),Du)- a(z_0,l(z_0),Dl))\cdot D\varphi
 -\partial _F a(z_0,l(z_0),Dl)(Du-Dl,D\varphi)]dz\nonumber\\
& +\Xint-_{Q_{\rho/2}^{(\lambda)}}( a(z,u,Du)- a(z,l(z_0),Du))\cdot D\varphi dz
 +\Xint-_{Q_{\rho/2}^{(\lambda)}}( a(z,l(z_0),Du)- a(z_0,l(z_0),Du))\cdot D\varphi dz\nonumber\\
&-\Xint-_{Q_{\rho/2}^{(\lambda)}}b(z,u,Du)\cdot\varphi dz
:=I_1+I_2+I_3+I_4.
\end{align}
Now, we start  to estimate $I_1$--$I_4$. For the term $I_1$, applying (\ref{2.04}), the  H\"{o}lder and Young's inequality, we have
\begin{align}\label{5.3}
|I_1|\leq&\Xint-_{Q_{\rho/2}^{(\lambda)}}\int_{0}^{1}\left|\partial_F a(z_0,l(z_0),Dl+s(Du-Dl))-\partial_F a(z_0,l(z_0),Dl)\right|ds|Du-Dl|dz\nonumber\\
\leq& c\Xint-_{Q_{\rho/2}^{(\lambda)}}\mu\left(\frac{|Du-Dl|}{1+|Dl|}\right)(1+|Dl|+|Du-Dl|)^{p-2}|Du-Dl|dz\nonumber\\
 \leq& c(1+|Dl|)^{p-1}\Xint-_{Q_{\rho/2}^{(\lambda)}}\mu\left(\frac{|Du-Dl|}{1+|Dl|}\right)\frac{|Du-Dl|}{1+|Dl|}dz\nonumber\\
 &+c(1+|Dl|)^{p-1}\Xint-_{Q_{\rho/2}^{(\lambda)}}\mu\left(\frac{|Du-Dl|}{1+|Dl|}\right)\frac{|Du-Dl|^{p-1}}{(1+|Dl|)^{p-1}}dz\nonumber\\ \leq&c(1+|Dl|)^{p-1}\left(\Xint-_{Q_{\rho/2}^{(\lambda)}}\mu^2\left(\frac{|Du-Dl|}{1+|Dl|}\right)dz\right)^{\frac{1}{2}}
 \left(\Xint-_{Q_{\rho/2}^{(\lambda)}}\frac{|Du-Dl|^2}{(1+|Dl)^2}dz\right)^{\frac{1}{2}}\nonumber\\
& + \nonumber c(1+|Dl|)^{p-1}\left(\Xint-_{Q_{\rho/2}^{(\lambda)}}\mu^{p}\left(\frac{|Du-Dl|}{1+|Dl|}\right)dz\right)^{\frac{1}{p}}
\left(\Xint-_{Q_{\rho/2}^{(\lambda)}}\frac{|Du-Dl|^p}{(1+|Dl)^p}dz\right)^{1-\frac{1}{p}}\nonumber\\
\leq& c(1+|Dl|)^{p-1}[\mu(\sqrt{E_\lambda(\rho)})+E_{\lambda}(\rho)]^{\frac12}\sqrt{E_\lambda(\rho)},
\end{align}
with $c=c(n,p,L,\nu)$, where in the last inequality, we have taken into account the Caccioppoli's type inequality (\ref{3.1}), Jensen's inequality for $\mu(\cdot)$ and $\mu^s(\cdot)\leq\mu(\cdot)$ for $s=2$ or $p$.

Likewise, applying (\ref{2.03}), Young's inequality, and note that $z\in Q_{\rho/2}^{(\lambda)}(z_0)$, then $I_2$ and  $I_3$ can be estimated as
\begin{align}\label{5.4}
  |I_2|+|I_3|\leq& L\Xint-_{Q_{\rho/2}^{(\lambda)}}[\omega(|u-l(z_0)|^2)+\omega(d(z,z_0)^2)](1+|Du|)^{p-1}dz\nonumber\\
\leq & c(p,L)\Xint-_{Q_{\rho/2}^{(\lambda)}}\left[\omega(|u-l(z_0)|^2)+\omega(\rho^2)\right]\left[(1+|Dl|)^{p-1}+|Du-Dl|^{p-1}\right]dz\nonumber\\
\leq & c(1+|Dl|)^{p-1}\left[\Xint-_{Q_{\rho/2}^{(\lambda)}}\omega (|u-l(z_0)|^2)dz+\omega(\rho^2)\right]\nonumber\\
&  + c(1+|Dl|)^{p-1}\Xint-_{Q_{\rho/2}^{(\lambda)}}[\omega (|u-l(z_0)|^2)+\omega(\rho^2)]\frac{|Du-Dl|^{p-1}}{(1+|Dl|)^{p-1}}dz\nonumber\\
\leq &c(1+|Dl|)^{p-1}\left[\Xint-_{Q_{\rho/2}^{(\lambda)}}\omega (|u-l(z_0)|^2)dz+\omega(\rho^2)\right]\nonumber\\
& +c(1+|Dl|)^{p-1}\left[\Xint-_{Q_{\rho/2}^{(\lambda)}}\omega (|u-l(z_0)|^2)+\omega(\rho^2)dz+\Xint-_{Q_{\rho/2}^{(\lambda)}}\frac{|Du-Dl|^{p}}{(1+|Dl|)^{p}}dz\right]\nonumber\\
\leq&  c(1+|Dl|^{p-1})E_\lambda(\rho),
\end{align}
where $c=c(n,p,\nu,L)$.

Taking into account the fact $\sup_{Q_{\rho/2}^{(\lambda)}(z_0)}|\varphi|\leq\rho\leq1$, similar with (\ref{4.92}),  we  infer that
\begin{align}\label{5.5}
|I_4|&\leq L\rho\Xint-_{Q_{\rho/2}^{(\lambda)}}(1+|Du|)^{p-1}dz+\rho\Xint-_{Q_{\rho/2}^{(\lambda)}}|u|^{q_1}dz\nonumber\\
&\leq \underbrace{c(L,p)\rho\Xint-_{Q_{\rho/2}^{(\lambda)}}(1+|Dl|)^{p-1}+|Du-Dl|^{p-1}dz}_{I_{41}}
+\rho\Xint-_{Q_{\rho/2}^{(\lambda)}}1+|u|^{\frac{n+2}{n}(p-1)}dz\nonumber\\
&\leq  I_{41}+\rho
+\frac{c(n,p)}{|Q_{\rho/2}^{(\lambda)}|}\rho\int_{\Lambda_{\rho/2}^{(\lambda)}}\|u\|^{\frac{n(n+2)(p-1)\theta_2}{n}}_{L^2(B_{\rho/2})}
(\|u\|^{\frac{n(n+2)(p-1)(1-\theta_2)}{n}}_{L^{p-1}(B_{\rho/2})}+\|Du\|^{\frac{n(n+2)(p-1)(1-\theta_2)}{n}}_{L^{p-1}(B_{\rho/2})})dt\nonumber\\
&\leq I_{41}+\rho+c(n,p)\sup_{t\in \Lambda_{\rho/2}^{(\lambda)}}\|u\|^{\frac{n(n+2)(p-1)\theta_2}{n}}_{L^2(B_{\rho/2})}\rho
\Xint-_{Q_{\rho/2}^{(\lambda)}}(|u|^{p-1}+|Du|^{p-1})dz\nonumber\\
&\leq I_{41}+c(n,p)\rho+c(n,p)\rho\Xint-_{Q_{\rho/2}^{(\lambda)}}|Du|^{p-1}dz\nonumber\\
&\leq c(L,n,p)(1+|Dl|)^{p-1}\rho+c(n,p)(1+|Dl|)^{p-1}\rho\Xint-_{Q_{\rho/2}^{(\lambda)}}|Du-Dl|^{p}dz\nonumber\\
&\leq c(1+|Dl|)^{p-1}E_\lambda(\rho),
\end{align}
where $c=c(n,p,L,\nu)$ and $\theta_2\in (0,1)$ is same with $\theta_1$ in (\ref{3.123}).

Plugging (\ref{5.3})-(\ref{5.5})  into (\ref{5.2}), then we have
\begin{equation*}
\Xint-_{Q_{\rho/2}^{(\lambda)}}[(u-l)\cdot \varphi_t-\partial_F a(z_0,l(z_0),Dl)(Du-Dl,D\varphi)]dz\leq c(1+|Dl|^{p-1})[E_\lambda(\rho)+\mu(\sqrt{E_\lambda(\rho)})]^{\frac12}\sqrt{E_\lambda(\rho)},
\end{equation*}
where $c=c(n,p,\nu,L)$. By scaling argument for general, then we have (\ref{5.1}).
\end{proof}
\subsection{Decay estimate}
The aim of this section is to provide a decay estimate of $\Phi_{\lambda_j}(z_0,\vartheta^{j}\rho,l_j)$ with $\lambda_j$, $\vartheta$, $l_j$ will be specified in later, from which we can obtain a Campanato type estimate of weak solution $u$ to (\ref{1.1}), then we deduce the regularity of $u$ by a standard argument of  Campanato space. First, we introduce  a standard estimate for weak solution to linear parabolic systems with constant coefficients (cf. \cite{A10} Lemma 5.1), which is necessary in the proof of decay estimate  of $\|u-(u)_{z_0,r}\|_{L^2(Q_r(z_0))}$.
\begin{Lemma}\label{le5.2}
Let $h\in L^2(\Lambda_\rho(t_0);W^{1,2}(B_\rho(x_0);\mathbb{R}^N))$ be a weak solution in $Q_\rho(z_0)$ of the following linear parabolic system with constant coefficients
\begin{equation}\label{5.05}
\Xint-_{Q_{\rho}(z_0)}(h\cdot\varphi_t-A(Dh,D\varphi))dz=0,
\end{equation}
for all $\varphi\in C_{0}^{\infty}(Q_\rho(z_0);\mathbb{R}^N)$, where the coefficients $A$ satisfy
\begin{equation*}
 A(F,F)\geq\nu|F|^2,\ \ \  A(F,\widetilde{F})\leq L |F||\widetilde{F}|,
\end{equation*}
for any $F,\widetilde{F}\in \mathbb{R}^{Nn}$. Then, $h$ is smooth in $Q_\rho(z_0)$ and for all $s\geq1$, $\theta\in(0,1]$, there holds
\begin{align*}
(\theta\rho)^{-s}\Xint-_{Q_{\theta\rho}(z_0)}&|h-(h)_{z_0,\theta\rho}-(Dh)_{z_0,\theta\rho}(x-x_0)|^sdz\nonumber\\
&\leq c_{pa}\theta^s\rho^{-s}\Xint-_{Q_{\rho}(z_0)}|h-(h)_{z_0,\rho}-(Dh)_{z_0,\rho}(x-x_0)|^sdz,
\end{align*}
for a constant $c_{pa}=c_{pa}(n,N,L/\nu)\geq1$.
\end{Lemma}
The  $A-$caloric  approximation lemma (Lemma \ref{le2.3}) allows one to translate these decay estimates on $h$ into a certain excess functional, e.g., $v$ in (\ref{5.012}). This eventually allows one to derive the partial regularity of $u$. Based on Lemma \ref{le5.1}-\ref{le5.2}, we have the following
result.
\begin{Lemma}\label{le5.3}(Decay estimate.)
Given  $\alpha\in (0,1)$ be a constant. Suppose  $H\geq1$ be a  constant, and $\rho_0=\rho_0(n,N,p,\nu,L,H,\alpha,\omega(\cdot),\mu(\cdot))\in (0,1]$. Let $u\in C^{0}(-T,0;L^2(\Omega;\mathbb{R}^N))\cap L^p(-T,0;W^{1,p}(\Omega;\mathbb{R}^N))$ is a weak solution to (\ref{1.1}) in $Q_T$ under the assumption (\ref{2.01})-(\ref{2.5})  and $Q_{\rho}^{(\lambda)}(z_0)\subset Q_T$ is a parabolic cylinder with reference point $z_0\in Q_T$, and $\rho\in(0,\rho_0]$. For the scaling factor $\lambda\geq1$, if there exist constants $\vartheta\in (0,1)$,  $\varepsilon_0=\varepsilon_0(n,N,p,\nu,L,H,\alpha,\mu(\cdot))\in (0,1)$ and $c_1=c_1(n,N,p,\nu,L,H)$, such that
\begin{equation}\label{5.7}
\lambda\leq1+|Dl_{z_0,\rho}^{(\lambda)}|\leq H\lambda
\end{equation}
and the smallness condition
\begin{equation}\label{5.8}
E_{\lambda}(z_0,\rho,l_{z_0,\rho}^{(\lambda)})\leq \varepsilon_0
\end{equation}
holds,  and for $\lambda=1$ on $Q_\rho(z_0)\equiv Q_{\rho}^{(1)}(z_0)$, there holds
\begin{align}\label{5.9}
\begin{cases}
 1+|Dl_{z_0,\rho}^{(1)}|\leq H,\\
   \Phi_1\left(l_{z_0},\rho,l_{z_0,\rho}^{(1)}\right)\leq \varepsilon_1, \\
    \end{cases}
\end{align}
with $\varepsilon_1=c_1\vartheta^2\varepsilon_0\leq\frac{\varepsilon_0}{3}$.
Then, there exist numbers $\{\lambda_j\}_{j=0}^{\infty}$ such that
\addtocounter{equation}{1}
\begin{equation}
\tag{A$_j$}\ \ \ \ \ \ \ \ \ \
\begin{cases}
1\leq \lambda_j\leq(2H)^j,  \\
\lambda_j\leq1+|Dl_{z_0,\vartheta^j\rho}^{(\lambda_j)}|\leq H\lambda_j,  \\
\Phi_{\lambda_j}\left(z_0,\vartheta^j\rho,l_{z_0,\vartheta^j\rho}^{(\lambda_j)}\right)\leq\varepsilon_1,\\
  \end{cases}
\end{equation}
and for any $r\in(0,\rho]$, there holds
\begin{equation}\label{5.10}
\int_{Q_r(z_0)}|u-(u)_{z_0,r}|^2 dz\leq cr^{n+2+2\alpha},
\end{equation}
where $c=c(n,N,p,\nu,L,H,\alpha)$.
\end{Lemma}
\begin{proof}
For the  convenience of notation, we shall once again omit the reference point $z_0=(x_0,t_0)$ in the notation, and we denote $Q_{\rho}^{(\lambda)}$, $l_j,l_{\rho}^{(\lambda)}$, $E_{\lambda}(\rho)$, $E_{\lambda_j}$, $ \Phi_{\lambda_j}(\vartheta^j\rho)$, $\Psi_{\lambda_j}(\vartheta^j\rho)$ instead of
 $Q_{\rho}^{(\lambda)}(z_0)$, $l_{z_0,\vartheta^{j}\rho}  ^{(\lambda_j)} (z_0)$, $l_{z_0,\rho}^{(\lambda)}(z_0)$, $E_{\lambda}(z_0,\rho,l_{\rho}^{(\lambda)})$,
 $E_{\lambda_j}(z_0,\vartheta^j\rho,l_j)$, $\Phi_{\lambda_j}(z_0,\vartheta^j\rho,l_j)$, $\Psi_{\lambda_j}(z_0,\vartheta^j\rho,l_j(z_0))$, respectively.

Suppose (A$_j$) holds, then a direct consequence of (A$_j$) is
\addtocounter{equation}{1}
\begin{equation}\tag{B$_j$}
 \Psi_{\lambda_j}(\vartheta^j\rho)\leq (2H)^{2(j+1)}(\vartheta^j\rho)^2.
\end{equation}
In fact, by (A$_j$), we have
\begin{align*}
\Psi_{\lambda_j}(\vartheta^j\rho)&=\Xint-_{Q_{\vartheta^j\rho}^{(\lambda_j)}}|u-l_j(z_0)|^2 dz\\
&\leq 2(\vartheta^j\rho)^2(1+|Dl_j|)^2\Phi_{\lambda_j}(\vartheta^j\rho)+2|l_j-l_j(z_0)|^2\\
&\leq 2(\vartheta^j\rho)^2(1+|Dl_j|)^2\Phi_{\lambda_j}(\vartheta^j\rho)+2(\vartheta^j\rho)^2|Dl_j|^2\\
&\leq 2^{2(j+1)}(\vartheta^j\rho)^2 H^{2(j+1)}.
\end{align*}
Before  proving (A$_j$), first, we propose to prove that from (\ref{5.7})--(\ref{5.8}) there exist $\lambda'\in [\frac{\lambda}{2},2H\lambda]$ such that
\begin{equation}\label{5.11}
1+|Dl_{\vartheta\rho}^{(\lambda')}|=\lambda'
\end{equation}
and there exists a constant  $c_1$  such that the decay estimate
\begin{equation}\label{5.12}
\Phi_{\lambda'}(\vartheta\rho)\leq c_1\vartheta^2 E_\lambda(\rho)
\end{equation}
holds. To prove (\ref{5.11}) and (\ref{5.12}), we first define
\begin{equation}\label{5.012}
v(x,t):=\frac{u(x,\lambda^{2-p}t)-l_{\rho}^{(\lambda)}(z)}{(1+|Dl_{\rho}^{(\lambda)}|) c_2\sqrt{E_\lambda(\rho)}},
\end{equation}
for all $(x,t)\in Q_\rho\equiv Q_{\rho}^{(1)}$ with $c_2\geq 1$ will be specified in later.
In virtue of  (\ref{3.1}) we can see that
\begin{equation}\label{5.13}
\Xint-_{Q_{\rho/2}^{(\lambda)}}\left[\frac{|Du-Dl_{\rho}^{(\lambda)}|^2}{(1+|Dl_{\rho}^{(\lambda)}|)^2}+\frac{|Du-Dl_{\rho}^{(\lambda)}|^p}{(1+|Dl_{\rho}^{(\lambda)}|)^p}\right]dz\leq cE_\lambda(\rho),
\end{equation}
where $c=c(n,p,\nu,L)$. From (\ref{5.8}), we define
\begin{equation*}
\gamma:=\sqrt{E_\lambda(\rho)}\leq 1.
\end{equation*}
Thus, by the aid of  (\ref{5.13}) and the definition of $E_\lambda(\rho)$, we deduce that
\begin{align}\label{5.14}
\Xint-_{Q_{\rho/2}^{(\lambda)}}&\left[\left|\frac{v}{\rho}\right|^2+|Dv|^2\right]dz+\gamma^{p-2}\Xint-_{Q_{\rho/2}^{(\lambda)}}
\left[\left|\frac{v}{\rho}\right|^p+|Dv|^p\right]dz\nonumber\\
 =&\frac{1}{E_\lambda(\rho)}\Xint-_{Q_{\rho/2}^{(\lambda)}}\frac{|u-l_{\rho}^{(\lambda)}|^2}
{c_{2}^{2}\rho^2(1+|Dl_{\rho}^{(\lambda)}|)^2}+\frac{|u-l_{\rho}^{(\lambda)}|^p}{c_{2}^{p}\rho^p(1+|Dl_{\rho}^{(\lambda)}|)^p}dz\nonumber\\
& +\frac{1}{E_\lambda(\rho)}\Xint-_{Q_{\rho/2}^{(\lambda)}}\frac{|Du-Dl_{\rho}^{(\lambda)}|^2}
{c_{2}^{2}(1+|Dl_{\rho}^{(\lambda)}|)^2}+\frac{|Du-Dl_{\rho}^{(\lambda)}|^p}
{c_{2}^{p}(1+|Dl_{\rho}^{(\lambda)}|)^p}dz \nonumber\\
\leq&\frac{2^{n+2}}{c_{2}^{2}}+\frac{c}{c_{2}^{2}}\leq 1
\end{align}
Indeed, we only need to choose $c_2=c_2(n,p,\nu,L)\geq1$ large enough, then the previous inequality is automatically satisfied.
Next, we define the bilinear form
\begin{equation*}
A(F,\widetilde{F}):=\frac{\partial_F a(z_0,l_{\rho}^{(\lambda)}(z_0),Dl_{\rho}^{(\lambda)})(F,\widetilde{F})}{\lambda^{p-2}}
\end{equation*}
for all $F,\widetilde{F}\in \mathbb{R}^{Nn}$. Taking into account (\ref{2.02}), (\ref{2.03}), and (\ref{5.7}), there holds
\begin{equation*}
A(F,F)\geq \nu|F|^2,\ \ \ \ \ \ \ \ \ \ \ A(F,\widetilde{F})\leq LH^{p-2}|F||\widetilde{F}|,
\end{equation*}
for all $F,\widetilde{F}\in \mathbb{R}^{Nn}$. Appealing to  (\ref{5.1}) and (\ref{5.7}), then we  have
\begin{align}\label{5.15}
\Xint-_{Q_{\rho/2}}(v\cdot\varphi_t-A(Dv,D\varphi))&
\leq \frac{c(1+|Dl^{(\lambda)}_{\rho}|)^{p-2}}{c_2\lambda^{p-2}}[E_\lambda(\rho)+\mu(\sqrt{E_\lambda(\rho)})]^{\frac{1}{2}}
 \sup_{Q_{\rho/2}^{(\lambda)}}|D\varphi|\nonumber\\
& \leq\frac{cH^{p-2}}{c_2}\left[E_\lambda(\rho)+\mu(\sqrt{E_\lambda(\rho)})\right]^{\frac{1}{2}} \sup_{Q_{\rho/2}^{(\lambda)}}|D\varphi|\nonumber\\
& \leq\left[E_\lambda(\rho)+\mu(\sqrt{E_\lambda(\rho)})\right]^{\frac{1}{2}} \sup_{Q_{\rho/2}^{(\lambda)}}|D\varphi|,
\end{align}
where in the last inequality, we have taken into account the fact $c_2\geq 1$ large enough.

Let $\varepsilon>0$ from Lemma \ref{le2.3}, which will be specified in later and $\delta\equiv\delta(n,p,\nu,LH^{p-2},\varepsilon)$ are constants from the $A-$caloric approximation Lemma \ref{le2.3}, here, we replace $L$ in Lemma \ref{le2.3} with $LH^{p-2}$.  From (\ref{5.14}), (\ref{5.15}) and the definition of $v$, $A$, we can see that  the all assumptions of Lemma \ref{le2.3} are satisfied, if we proved the smallness condition
\begin{align}\label{5.16}
\left[E_\lambda(\rho)+\mu(\sqrt{E_\lambda(\rho)})\right]^\frac{1}{2}\leq \delta
\end{align}
holds. Thus, applying Lemma \ref{le2.3}, there exists a $A-$caloric function $h\in L^2(\Lambda_{\rho/4};W^{1,2}(B_{\rho/4};\mathbb{R}^N))$ on $Q_{\rho/4}$, such that
\begin{equation}\label{5.17}
\Xint-_{Q_{\rho/4}}\left[\left|\frac{h}{\rho/4}\right|^2+|Dh|^2\right]dz+\gamma^{p-2}\Xint-_{Q_{\rho/4}}\left[\left|\frac{h}{\rho/4}\right|^p+|Dh|^p\right]dz\leq 2^{n+3+2p},
\end{equation}
and
\begin{equation}\label{5.18}
\Xint-_{Q_{\rho/4}}\left[\left|\frac{v-h}{\rho/4}\right|^2+\gamma^{p-2}\left|\frac{v-h}{\rho/4}\right|^p\right]dz\leq \varepsilon.
\end{equation}
Taking into account Lemma \ref{le5.2} and (\ref{5.17}), for $\theta\in (0,\frac{1}{4}]$, $s=2$ and $s=p$, there exists a constant
$c_{pa}\equiv c_{pa}(n,N,\nu,LH^{p-2})=c_{pa}(n,N,\nu,L,p,H)$,
such that for the $A-$caloric function $h$ satisfies
\begin{align}\label{5.19}
(\theta\rho)^{-s}&\Xint-_{Q_{\theta\rho}}|h-(h)_{\theta\rho}-(Dh)_{\theta\rho}x|^s dz\nonumber\\
& \leq c_{pa}\theta^{s}(\rho/4)^{-s}\Xint-_{Q_{\rho/4}}|h-(h)_{\rho/4}-(Dh)_{\rho/4}x|^sdz\nonumber\\
&\leq 3^{s-1}c_{pa}\theta^s\left[(\rho/4)^{-s}\Xint-_{Q_{\rho/4}}(|h|^s+|(h)_{\rho/4}|^s)dz+|(Dh)_{\rho/4}|^s\right]\nonumber\\
&\leq 2\cdot3^{s-1}c_{pa}\theta^s\left[(\rho/4)^{-s}\Xint-_{Q_{\rho/4}}|h|^s dz+\Xint-_{Q_{\rho/4}}|Dh|^s dz\right]\nonumber\\
&\leq 2^{n+2p+4}\cdot3^{p-1}c_{pa}\gamma^{2-s}\theta^s.
\end{align}
Employing (\ref{5.18}) and (\ref{5.19}), for all $\theta\in (0,\frac{1}{4}]$, we further obtain
\begin{align}\label{5.20}
(\theta\rho)^{-s}&\Xint-_{Q_{\theta\rho}}|v-(h)_{\theta\rho}-(Dh)_{\theta\rho}x|^s dz\nonumber\\
& \leq 2^{s-1}\left[(\theta\rho)^{-s}\Xint-_{Q_{\theta\rho}}|v-h|^s dz+(\theta\rho)^{-s}\Xint-_{Q_{\theta\rho}}|h-(h)_{\theta\rho}-(Dh)_{\theta\rho} x|^sdz \right]\nonumber\\
&\leq 2^{s-1}\left[(4\theta)^{-n-2-s}(\rho/4)^{-s}\Xint-_{Q_{\rho/4}}|v-h|^s dz+2^{n+2p+4}3^{p-1}c_{pa}\gamma^{2-s}\theta^s\right]\nonumber\\
&\leq 2^{n+4p+2}c_{pa}\gamma^{2-s}[\theta^{-n-2-s}\varepsilon+\theta^s].
\end{align}
Now, we choose $\varepsilon:=\theta^{n+2+2s}$ with $\theta\in (0,\frac{1}{4}]$ is a fixed parameter will be specified in later, at this stage, we have also determined the constant $\delta\equiv\delta(n,p,\nu,L,H,\varepsilon)=\delta(n,p,\nu,L,H,\theta)$ in (\ref{5.16}).

From the definition of $v$ and (\ref{5.7}), (\ref{5.20}), by scaling back, for $s=2$ or $s=p$, there holds
\begin{align}\label{5.21}
(\theta\rho)^{-s}&\Xint-_{Q_{\theta\rho}^{(\lambda)}}|u-l_{\rho}^{(\lambda)}-(1+|Dl^{\lambda}_{\rho}|) c_2\gamma((h)_{\theta\rho}^{(\lambda)}+(Dh)_{\theta\rho}^{(\lambda)}x)|^sdz\nonumber\\
&\leq 2^{n+4p+3}c_{pa}H^s\lambda^s\theta^s\gamma^2c_{2}^{s}=c\theta^s\lambda^s E_\lambda(\rho),
\end{align}
with $c=c(n,N,p,\nu,L,H)$.

As a consequence,  from  the minimizing property of $l_{\theta\rho}^{(\lambda)}$ in Lemma \ref{le2.2} and (\ref{5.21}), it follows that for $s=2$ or $s=p$
\begin{equation}\label{5.22}
\Xint-_{Q_{\theta\rho}^{(\lambda)}}\left|\frac{u-l_{\theta\rho}^{(\lambda)}}{\theta\rho}\right|^s dz\leq c_3\theta^s\lambda^s E_\lambda(\rho),
\end{equation}
with $c_3=c_3(n,N,p,\nu,L,H)$.

Now, we concentrate our attention on the proof of (\ref{5.11}) and (\ref{5.12}). Define
\begin{equation*}
\vartheta:=2^{-\frac{p-2}{2}}\theta,
\end{equation*}
which implies $\vartheta\in (0,2^{-\frac{p+2}{2}}]$ due to $\theta\in(0,\frac{1}{4}]$. For some $\bar{\mu}\in(\frac{\lambda}{2},2H\lambda]$, then we have
$Q_{\vartheta\rho}^{(\bar{\mu})}\subset Q_{\theta\rho}^{(\lambda)}$. Hence, by $(\ref{5.22})$, there holds
\begin{align}\label{5.23}
|Dl_{\vartheta\rho}^{(\bar{\mu})}-Dl_{\rho}^{(\lambda)}|^2&
\leq\frac{n(n+2)}{(\vartheta\rho)^2}\Xint-_{Q_{\vartheta\rho}^{(\bar{\mu})}}|u-l_{\rho}^{(\lambda)}(x_0)-Dl_{\rho}^{(\lambda)}(x-x_0)|^2dz\nonumber\\
&=n(n+2)\Xint-_{Q_{\vartheta\rho}^{(\bar{\mu})}}\left|\frac{u-l_{\rho}^{(\lambda)}}{\vartheta\rho}\right|^2 dz
\leq n(n+2)\vartheta^{-(n+4)}\left(\frac{\lambda}{\bar{\mu}}\right)^{2-p}\Xint-_{Q_{\rho}^{(\lambda)}}\left|\frac{u-l_{\rho}^{(\lambda)}}{\rho}\right|^2 dz\nonumber\\
&\leq n(n+2)\vartheta^{-(n+4)}(2H)^{p-2}\lambda^2E_\lambda(\rho)\nonumber\\
&\leq c_4\lambda^2E_\lambda(\rho)\leq\frac{\lambda^2}{4},
\end{align}
where in the first inequality we have used (\ref{2.4}) with $\xi\equiv l_{\rho}^{(\lambda)}(z_0)$, $w=Dl_{\rho}^{(\lambda)}$ on $Q_{\vartheta\rho}^{(\bar{\mu})}$  and in the last inequality, we have taken into account smallness assumption of $E_\lambda(\rho)$, that is
\begin{equation}\label{5.24}
c_4E_\lambda(\rho)\leq\frac{1}{4},
\end{equation}
with $c_4=c_4(n,\vartheta,p,H)$.

Applying (\ref{5.7}), (\ref{5.23}) we can see that
\begin{equation}\label{5.25}
1+\left|Dl_{\vartheta\rho}^{(\bar{\mu})}\right|
\leq 1+\left|Dl_{\rho}^{(\lambda)}\right|+\left|Dl_{\vartheta\rho}^{(\bar{\mu})}-Dl_{\rho}^{(\lambda)}\right|\leq H\lambda+\frac{\lambda}{2}
\leq 2H\lambda,\\
\end{equation}
and
\begin{equation}\label{5.26}
1+\left|Dl_{\vartheta\rho}^{(\bar{\mu})}\right|\geq 1+\left|Dl_{\rho}^{(\lambda)}\right|-\left|Dl_{\vartheta\rho}^{(\bar{\mu})}-Dl_{\rho}^{(\lambda)}\right|\geq\lambda-\frac{\lambda}{2}\geq\frac{\lambda}{2}.
\end{equation}
Define $f:[\frac{\lambda}{2},2H\lambda]\rightarrow \mathbb{R} $ by
\begin{equation*}
f(\bar{\mu}):=\bar{\mu}-(1+|Dl_{\vartheta\rho}^{(\bar{\mu})}|),
\end{equation*}
then $f$ is a continuous function. Appealing to  (\ref{5.25})-(\ref{5.26}), we have $f(\frac{\lambda}{2})\leq 0$ and $f(2H\lambda)\geq 0$.
Thus, there exists $\lambda'\in[\frac{\lambda}{2},2H\lambda]$ such that $f(\lambda')=0$, that is $\lambda'=1+Dl_{\vartheta\rho}^{(\lambda')}$,  whence (\ref{5.11}).

Next, for $s=2$ or $s=p$,  we once again using the minimizing property of $l$,  (\ref{5.11}), (\ref{5.22}) and the definition of $\vartheta$, we deduce that
\begin{align*}
\Xint-_{Q_{\vartheta\rho}^{(\lambda')}}\left|\frac{u-l_{\vartheta\rho}^{(\lambda')}}{\vartheta\rho}\right|^s dz&\leq c(n,p)\Xint-_{Q_{\vartheta\rho}^{(\lambda')}}\left|\frac{u-l_{\theta\rho}^{(\lambda)}}{\vartheta\rho}\right|^s dz\nonumber\\
&\leq c\left(\frac{\theta}{\vartheta}\right)^{n+2+s}\left(\frac{\lambda}{\lambda'}\right)^{2-p}\Xint-_{Q_{\theta\rho}^{(\lambda)}}
\left|\frac{u-l_{\theta\rho}^{(\lambda)}}{\theta\rho}\right|^s dz\nonumber\\
&\leq c c_3\left(\frac{\theta}{\vartheta}\right)^{n+2+s}(2H)^{p-2}\theta^s\lambda^s E_\lambda(\rho)\nonumber\\
&\leq c c_3\left(\frac{\theta}{\vartheta}\right)^{n+2+s}(2H)^{p+s-2}\theta^s(\lambda')^s E_\lambda(\rho)\nonumber\\
&\leq c_1\vartheta^s\left(1+\left|Dl_{\vartheta\rho}^{(\lambda')}\right|\right)^s E_\lambda(\rho),
\end{align*}
with $c_1=c_1(n,N,p,\nu,L,H)$. Whence (\ref{5.12}).

From now on, we have determined $\delta\equiv\delta(n,p,\nu,L,H,\theta=2^{\frac{p-2}{2}}\vartheta$)=$\delta(n,p,\vartheta,L,H,\nu)$ in (\ref{5.16}),
and hence from (\ref{5.16}) and (\ref{5.24}) we have also determined $\varepsilon_0=\varepsilon_0(n,p,\nu,L,$ $H,\vartheta,\mu(\cdot))$,  $c_1=c_1(n,N,p,\nu,L,$ $H)$, which is close to fulfilled the conditions of the Lemma \ref{le5.3}, it  remains to determine the constant $\vartheta\in (0,1)$. For $\alpha\in (0,1)$, let
\begin{equation}\label{5.28}
\vartheta:=min \left\{\left(\frac{1}{2}\right)^\frac{p+2}{2},\left(\frac{1}{3c_1}\right)^\frac{1}{2},\left(\frac{1}{2H}\right)^{\frac{p(n+4)}{4(1-\alpha)}}\right\}\\
\end{equation}
and
\begin{equation}\label{5.29}
\varepsilon_1:=c_1(\vartheta)^s\varepsilon_0.
\end{equation}
Joining (\ref{5.28}) with (\ref{5.29}), we can see that once $\vartheta$ is chosen, which is dependent on $n,N,p,\nu,L,\alpha,H$, then $\varepsilon_0$ from (\ref{5.8}) is determined, that is $\varepsilon_0=\varepsilon_0(n,N,p,\nu,L,\alpha,H)$. Moreover, there holds
 $\varepsilon_1\leq \frac{\varepsilon_0}{3}$.

According to the conclusion above, now, we focus our attention on  proving (A$_j$). We shall use the induction argument, first, consider the case (A$_0$). Taking into account  the assumption (\ref{5.9}), let $\lambda_0=1$, we have  (A$_0$) holds. We now choose $\rho_0=\rho_0(n,N,p,\nu,L,H,\alpha,\omega(\cdot),\mu(\cdot))\in (0,1]$ suitable small such that
\begin{equation}\label{5.30}
\omega((2H\rho_0)^2)+2H\rho_0\leq\varepsilon_1,
\end{equation}
and suppose that (A$_j$) holds for some $j\in\{0,1,2\cdots\}$, we proceed to prove (A$_{j+1}$) holds. By claimed as before, from (A$_j$) we have (B$_j$) holds, then using  the assumptions  (A$_j$) and (\ref{5.28}) we deduce that
\begin{equation*}
\Psi_{\lambda_j}(\vartheta^j\rho)\leq (2H)^{2(j+1)}(\vartheta^j\rho)^2\leq (2H\rho)^2\leq (2H\rho_0)^2.
\end{equation*}
Thus, in virtue of (\ref{5.30}), we infer that
\begin{equation}\label{5.31}
\omega(\Psi_{\lambda_j}(\vartheta^j\rho))\leq\omega((2H\rho_0)^2)\leq\varepsilon_1.
\end{equation}
Similarly, by $(\ref{5.30})$, we also have
\begin{equation}\label{5.32}
\omega((\vartheta^j\rho)^2)+(\vartheta^j\rho)\leq \omega(\rho^2)+\rho\leq \omega((2H\rho_0)^2)+2H\rho_0\leq\varepsilon_1.
\end{equation}
Taking into account (\ref{5.31}), (\ref{5.32}) and the induction assumption (A$_j$)$_3$, we infer that
\begin{equation}\label{5.33}
E_{\lambda_ j}(\vartheta^j\rho)=\Phi_{\lambda _j}(\vartheta^j\rho)+\omega(\Psi_{\lambda_ j}(\vartheta^j\rho))+\omega((\vartheta^j\rho)^2)+\vartheta^j\rho\leq 3\varepsilon_1\leq \varepsilon_0.
\end{equation}
Finally, by (\ref{5.33}) and the induction  assumption (A$_j$)$_2$, we can replace $(\rho,\lambda)$ in (\ref{5.7}), (\ref{5.8}) by $(\vartheta^j\rho,\lambda_j)$, then, from (\ref{5.11}), (\ref{5.12}) and (\ref{5.25}), there exists a number $\lambda_{j+1}\in\left[\frac{\lambda_j}{2},2H\lambda_j\right]$ such that
  \begin{equation}\label{5.34}
1+|Dl_{j+1}|=\lambda_{j+1}\ \  \  \ \
\end{equation}
and
\begin{equation}\label{5.35}
\Phi_{\lambda_{j+1}}(\vartheta\vartheta^j\rho)\leq c_1\vartheta^2E_{\lambda_j}(\vartheta^j\rho)\leq\varepsilon_1.
\end{equation}
In view of (\ref{5.34}) and (\ref{5.35}) we have (A$_{j+1}$)$_2$, (A$_{j+1}$)$_3$ hold. Furthermore, applying  (A$_j$)$_1$, we refer that (A$_{j+1}$)$_1$ holds since $\lambda_{j+1}\leq 2H\lambda_j$. Thus, we have proved (A$_{j+1}$). For simplicity, here, we represent the recursive relationship as follows
\begin{equation}\label{5.035}
E_{\lambda}\longrightarrow \Phi_{\lambda_1}\longrightarrow A_1\longrightarrow B_1\longrightarrow E_{\lambda_1}\longrightarrow\cdots\longrightarrow
A_j\longrightarrow B_j\longrightarrow E_{\lambda_j}\longrightarrow\cdots
\end{equation}
Based on (\ref{5.035}), now, it   remains to prove (\ref{5.10}). First, from {(B$_j$)}, we can see that for some $j\in \{1,2\cdots\}$
\begin{align}\label{5.36}
\int_{Q_{\vartheta^j\rho}^{(\lambda_j)}}\left|u-(u)_{z_0,\vartheta^j\rho}^{(\lambda_j)}\right|^2 dz&\leq \alpha_n(\lambda_{j}^{2-p})(2H)^{2(j+1)}(\vartheta^j\rho)^{n+4}\nonumber\\
& \leq \alpha_n (2H)^{2(j+1)}(\vartheta^j\rho)^{n+4},
\end{align}
where $\alpha_n$ denotes the volume of unit ball in $\mathbb{R}^n$ and in the previous inequality we have used the fact $\xi_{z_0,\rho}^{(\lambda)}=(u)_{z_0,\rho}^{(\lambda)}$ in (\ref{2.3}).

Now, we  define
\begin{equation*}
\theta:=(2H)^{\frac{2-p}{2}}\vartheta,
\end{equation*}
which implies the inclusion $Q_{\theta^{j}\rho}(z_0)\subset Q_{\vartheta^j \rho}^{(\lambda_{j})}(z_0)$, using  (\ref{5.36}), (\ref{5.28}), and the minimizing property of $(u)_{z_0, \vartheta^{j}\rho}$ we infer that
\begin{align}\label{5.38}
\int_{Q_{\theta^{j}\rho}}\left|u-(u)_{z_0,\theta^{j}\rho}\right|^2 dz&\leq \int_{Q_{\theta^{j}\rho}(z_0)}\left|u-(u)_{z_0,\vartheta^j\rho}^{(\lambda_j)}\right|^2 dz
\leq\int_{Q_{\vartheta^{j}\rho}^{(\lambda_j)}(z_0)}|u-(u)_{z_0,\vartheta^j\rho}|^2dz\nonumber\\
&\leq \alpha_n(2H)^{2(j+1)}\theta^{j(n+2+2\alpha)}\left[(2H)^\frac{(p-2)(n+2+2\alpha)}{2}\vartheta^{2-2\alpha}\right]^j\rho^{n+4}\nonumber\\
&\leq \alpha_n(2H)^{2}\theta^{j(n+2+2\alpha)}\left[(2H)^\frac{p(n+4)}{2}\vartheta^{2-2\alpha}\right]^j\rho^{n+4}\nonumber\\
&\leq\alpha_n(2H)^{2}\theta^{j(n+2+2\alpha)}\rho^{n+4}.
\end{align}
For any $r\in(0,\rho]$, there exists $j\in\{0,1,2,\cdots\}$ such that $\theta^{j+1}\rho<r\leq \theta^{j}\rho$, making use of (\ref{5.38}), we obtain
\begin{align*}
\int_{Q_{r}(z_0)}|u-(u)_{z_0,r}|^2 dz&\leq \int_{Q_{r}(z_0)}|u-(u)_{z_0,\vartheta^{j}\rho}|^2 dz\nonumber\\
&\leq \int_{Q_{\theta^j\rho}(z_0)}|u-(u)_{z_0,\vartheta^{j}\rho}|^2 dz\nonumber\\
&\leq \alpha_n(2H)^{2}\theta^{j(n+2+2\alpha)}\rho^{n+4}\nonumber\\
&\leq\alpha_n\theta^{-(n+2+2\alpha)}\left(\frac{r}{\rho}\right)^{n+2+2\alpha}(2H)^{2}\rho^{n+4}\nonumber\\
&\leq cr^{n+2+2\alpha},
\end{align*}
with $c=c(n,N,p,\nu,L,H,\alpha)$, whence $(\ref{5.10})$.
\end{proof}
\begin{Remark}
Recalling Lemma \ref{le3.1} and (\ref{5.38}), we can see that, in the whole paper, we only need $\rho\in (0,1)$ suitable small, and
it does not tend to zero as $j\longrightarrow\infty$.
\end{Remark}
From Lemma \ref{le4.1} and Lemma \ref{le5.3}, now, we are able to prove Theorem \ref{th1.1}.
\begin{proof}[Proof of Theorem \ref{th1.1}]
Let $z_0\in Q_T \backslash (\Sigma_1\cup\Sigma_2)$, then by the definition of $\Sigma_1\ $and $\Sigma_2$, there exist some constants $ \varepsilon_2\in(0,1]$, $M_0\geq1$ such that
\begin{equation}\label{5.39}
\Xint-_{Q_{\rho}(z_0)}|Du-(Du)_{z_0,\rho}|^p dz\leq\varepsilon_2\\
\end{equation}
and
\begin{equation}\label{5.40}
|(Du)_{z_0,\rho}|\leq M_0.
\end{equation}
Now, in virtue of  (\ref{2.3}), (\ref{5.39}), (\ref{5.40}) and (\ref{4.1}), we are in a position to obtain
\begin{align*}
|Dl_{z_0,\rho}|&=\frac{n+2}{\rho^2}\left|\Xint-_{Q_{\rho}(z_0)}(u-(u)_{z_0,\rho})\otimes(x-x_0)dz\right|\\
&\leq c\Xint-_{Q_{\rho}(z_0)}(1+|Du|)^p dz\\
&\leq c[(1+M_0)^p+\varepsilon_2],
\end{align*}
where $c=c(n,N,p,L)$.

Since $\varepsilon_2\leq1\leq M_0$, then the previous inequality implies that
\begin{equation}\label{5.42}
|Dl_{z_0,\rho}|\leq c(n,N,p,L)M_0^{p}.
\end{equation}
Furthermore, by the minimality of $l_{z_0,\rho}$, (\ref{4.2}) and (\ref{5.39})--(\ref{5.40}), we can see that
\begin{align*}
\Xint-_{Q_\rho(z_0)}\left|\frac{u-l_{z_0,\rho}}{\rho}\right|^pdz
\leq& c(n,p)\Xint-_{Q_{\rho}(z_0)}\left|\frac{u-(u)_{z_0,\rho}-(Du)_{z_0,\rho}(x-x_0)}{\rho}\right|^p dz\nonumber\\
\leq& c\varepsilon_2[(1+M_0)^p+\varepsilon_2]^{p-2}+c(1+M_0)^{p(p-1)}[\omega(\rho^2)]^p[1+M_{0}^{p}+\varepsilon_2]^{2p}\nonumber\\
&+c\rho^p[(1+M_0)^{p^2}+\varepsilon_{2}^{p}].
\end{align*}
This implies  that
\begin{equation}\label{5.46}
\Xint-_{Q_\rho(z_0)}\left|\frac{u-l_{z_0,\rho}}{\rho}\right|^p dz\leq c\varepsilon_2 M_{0}^{p(p-2)}+c[\omega(\rho^2)]^p M_{0}^{2p^3(p-1)}+c\rho^p M_{0}^{p^2},
\end{equation}
with $c=c(n,N,p,L)$.

Appealing to  $(\ref{5.42})$-$(\ref{5.46})$, for suitable small $\varepsilon_2\in (0,1)$, we can deduce the existence of $H\geq 1$ and  $0<\rho\leq\rho_0(H)$ such that $Q_{2\rho}(z_0)\subset Q_T$, and at this stage, we further obtain
\begin{equation*}
1+|Dl_{z_0,\rho}|<H,
\end{equation*}
and
\begin{equation*}
\Phi_1(z_0,\rho,l_{z_0,\rho})<\varepsilon_1(H).
\end{equation*}
Note that the mappings
\begin{equation*}
z\mapsto Dl_{z,\rho},\ \  \  \  \    z\mapsto \Phi_1(z,\rho,l_{z,\rho})
\end{equation*}
are continuous. Thus, there exists $0<R\leq \frac{\rho}{2}$ such that
\begin{equation*}
  1+|Dl_{z,\rho}|<H,
\end{equation*}
and
\begin{equation*}
\Phi_1(z,\rho,l_{z,\rho})<\varepsilon_1(H),
\end{equation*}
for all $z\in Q_{R}(z_0)$. Hence, we infer for the suitable $\rho_0(H)\in(0,1]$, the smallness condition of Lemma \ref{le5.3}  is uniformly hold for  $z\in Q_{R}(z_0)$. Observe that, for all $z\in Q_{R}(z_0)$, there holds $Q_\rho(z)\subset Q_{2\rho}(z_0)\subset Q_T$, then for all $r\in(0,\rho)$ and $z\in Q_R(z_0)$ we obtain
\begin{equation*}
\int_{Q_r(z)}|u-(u)_{z,r}|^2 dz\leq cr^{n+2+2\alpha},
\end{equation*}
with $c=c(n,N,p,\nu,L,H,\alpha)$. By the Campanato space argument (cf. \cite{A15,4A}), we have $u\in C^{0;\alpha,\alpha/2}$ in a neighborhood of any point $z_0\in Q_T \backslash(\Sigma_1\cup\Sigma_2)$, and we further obtain $\left|\Sigma_1\cup\Sigma_2\right|=0$, which means $|Q_0|=|Q_T|$.
\end{proof}
\section{Estimate of singular set}\label{se6}
In this section, with Theorem \ref{th1.1} in hand, now, we proceed to prove Theorem \ref{th1.2}.
Such result will be proved by combining the  fractional time and  fractional space differentiability of gradient of weak solution $u$ to (\ref{1.1}).
\subsection{Fractional time differentiability}\label{se6.1}
In this subsection, we aim to proving the fractional time differentiability of $Du$ for  $p=2$.
First, we estimate the $L^2$-norm of $\tau^h u$.
\begin{Lemma}\label{hdle3.1}
Let $u\in L^{\infty}(-T,0;L^2(\Omega;\mathbb{R}^N))\cap L^2(-T,0;H^1(\Omega;\mathbb{R}^N))$ be a weak solution to (\ref{1.1}).
Let $(t_0,t_1)\subset\subset (-T,0)$ and $\eta\in C_0^{\infty}(\Omega)$ be a cut-off function with $supp (\eta)\subset\subset \Omega$. Then, whenever
$0<|h|\leq\frac{1}{2}\min\{|t_1|, T-|t_0|,1\}$, the following estimate holds
\begin{equation}\label{hd3.1}
\int_{t_0}^{t_1}\int_{\Omega}\eta^2|\tau^hu(x,t)|^{2}dxdt\leq c(|h|\|\eta\|_{L_x^{\infty}}^2+|h|\|\nabla\eta\|_{L_x^{\infty}}^2)\int_{Q_T}(1+|Du|^2)dxdt,
\end{equation}
where $c=c(L)$.
\end{Lemma}
\begin{proof}
First, we restrict ourselves to  the case $h>0$,   choosing  $\eta^2 \tau^h u$ as a test function in the Steklov averages formulation of (\ref{hd2.2}),
and integrating with respect to $t\in (t_0, t_1)$, we deduce
\begin{align}\label{hd3.2}
\int_{t_0}^{t_1 }\int_{\Omega}\eta^2 \frac{|\tau^hu|^{2}}{h}dxdt=&-  \int_{t_0}^{t_1 }\int_{\Omega}[a(z,u,Du)]_h\cdot(2\eta\nabla \eta\otimes \tau^h u+\eta^2D\tau^h u)dxdt\nonumber\\
& + \int_{t_0}^{t_1 }\int_{\Omega} [b(z,u,Du)]_h\cdot\eta^2 \tau^h udxdt \nonumber\\
:=&I_1+I_2+I_3.
\end{align}
Taking into account ($H_1$), ($H_3$) and Young's inequality, there holds
\begin{equation}\label{hd3.3}
|I_1|\leq \frac{1}{4} \int_{t_0}^{t_1 }\int_{\Omega}\eta^2 \frac{|\tau^hu|^{2}}{h}dxdt+4L^2\|\nabla \eta\|_{L_x^{\infty}}^2h\int_{Q_T}(1+|Du|^2)dxdt.
\end{equation}
For the term $I_2$, from the H\"{o}lder's inequality, it follows that
\begin{align}\label{hd3.4}
|I_2|\leq &\left(\int_{t_0}^{t_1}\int_{\Omega}\eta^2|[a(z,u,Du)]_h|^2dxdt\right)^{\frac{1}{2}}\left(\int_{t_0}^{t_1}\int_{\Omega}\eta^2|D\tau^h u|^2dxdt\right)^{\frac{1}{2}}\nonumber\\
\leq&2L\left(\|\eta\|_{L_x^{\infty}}^2\int_{t_0}^{t_1+h}\int_{\Omega}(1+|Du|^2)dxdt\right)^{\frac{1}{2}}
\left(\|\eta\|_{L_x^{\infty}}^2\int_{t_0}^{t_1+h}\int_{\Omega}|D u|^2dxdt\right)^{\frac{1}{2}}\nonumber\\
\leq& 2L\|\eta\|_{L_x^{\infty}}^2\int_{Q_T}(1+|Du|^2)dxdt.
\end{align}
Similarly,  for $p=2$, applying (\ref{4.92})-(\ref{4.96}), the term $I_3$ can be estimated as
\begin{align}\label{hd3.5}
|I_3|&\leq \int_{t_0}^{t_1}\int_{\Omega}[(1+|Du|)+|u|^{q_1}]\cdot\eta^2\tau^{h}udxdt\nonumber\\
&\leq c(L)\|\eta\|_{L_x^{\infty}}^2h\int_{Q_T}(1+|Du|^2)dxdt+\frac14\int_{t_0}^{t_1}\int_{\Omega}\frac{\eta^2|\tau^h u|^2}{h}dxdt.
\end{align}
Finally, we note that the estimation in the other one being the same using $u_{\bar{h}}$ instead of $u_h$. Now, inserting (\ref{hd3.3})-(\ref{hd3.5}) into
(\ref{hd3.2}), we  obtain
\begin{equation*}
\int_{t_0}^{t_1}\int_{\Omega}\eta^2\frac{|\tau^hu|^2}{|h|}dxdt\leq c(\|\eta\|_{L_x^{\infty}}^2+\|\nabla\eta\|_{L_x^{\infty}}^2)
\int_{Q_T}(1+|Du|^2)dxdt,
\end{equation*}
where $c=c(L)$.  Thus, we have (\ref{hd3.1}).
\end{proof}
From Lemma \ref{hdle3.1},  we have a direct result:
\begin{Remark}
Let $u\in L^{\infty}(-T,0;L^2(\Omega;\mathbb{R}^N))\cap L^2(-T,0;H^1(\Omega;\mathbb{R}^N))$ be a weak solution to (\ref{1.1}). Let
$(t_0,t_1)\subset\subset (-T,0)$ and $\Omega'\subset\subset \Omega$. Then, whenever
$0<|h|\leq\frac{1}{2}\min\{|t_1|, T-|t_0|,1\}$, there holds
\begin{equation}\label{hd3.7}
\int_{t_0}^{t_1}\int_{\Omega'}|\tau^hu(x,t)|^2dxdt\leq c|h|\int_{Q_T}(1+|Du|^2)dxdt,
\end{equation}
where $c=c(L,dist(\Omega',\partial\Omega))$.
\end{Remark}
Based on Lemma \ref{hdle3.1}, we now propose to estimate the time derivative of $Du$, which will be as the starting point of an iteration process.
\begin{Lemma}\label{hdle3.2}
Let $u\in L^{\infty}(-T,0;L^2(\Omega;\mathbb{R}^N))\cap L^2(-T,0;H^1(\Omega;\mathbb{R}^N))$ be a weak solution to (\ref{1.1}).
Let $(t_0,t_1)\subset\subset (-T,0)$ and $\eta\in C_0^{\infty}(\Omega)$ be a cut-off function with $supp (\eta)\subset\subset \Omega$. Then, there exists
a constant $c=c(L,\nu,dist(supp (\eta),\partial\Omega), t_0, t_1)$, such that whenever
$0<|h|< \min\{|t_1|,T-|t_0|,1, dist(supp(\eta),\partial\Omega)\}$, there holds
\begin{equation}\label{hd3.8}
\int_{t_0}^{t_1}\int_{\Omega}|\tau^{-h}Du|^2\eta^2dxdt\leq c|h|^{\frac{1}{2}}\int_{Q_T}(1+|Du|^2)dxdt.
\end{equation}
\end{Lemma}
\begin{proof}
We choose  $\chi_{\varepsilon}:\mathbb{R}\longmapsto [0,1]$ be a continuous affine function satisfying
\begin{align*}
\chi_{\varepsilon} (t)=
\begin{cases}
 1&t\in (-\infty,l), \\
 \frac{l+\varepsilon-t}{\varepsilon} & t\in [l,l+\varepsilon],\\
 0&t\in (l+\varepsilon,\infty),
  \end{cases}
\end{align*}
which approximating the characteristic function of $(-\infty,l)$ with $l\in (-T,0)$. Let $\xi(t)$, $\eta (x)$ are cut-off functions in the time and space
variables, respectively, such that
\begin{align*}
\begin{cases}
\xi(t)=1 & t\in (t_0,t_1), \\
supp (\xi) \subset (\tilde{t}_0,\tilde{t}_1)& \tilde{t}_0=\frac{-T+t_0}{2},\ \tilde{t}_1=\frac{t_1}{2}.
\end{cases}
\end{align*}
We take $\phi(x,t)=\tau^{h}(\eta^2(x)\xi^2(t)\chi_{\varepsilon}(t)\tau^{-h}u_{\lambda}(x,t))$ as a test function in (\ref{hd2.2}), and integrating in time
respect to $(-T,0)$, yields that
\begin{align}\label{hd3.9}
&\int_{-T}^{0}\int_{\Omega}\partial_tu_{\lambda}\cdot\tau^{h}(\eta^2(x)\xi^2(t)\chi_{\varepsilon}(t)\tau^{-h}u_{\lambda}(x,t))\nonumber\\
&+[a(z,u,Du)]_{\lambda}\cdot D\tau^{h}(\eta^2(x)\xi^2(t)\chi_{\varepsilon}(t)\tau^{-h}u_{\lambda}(x,t))\nonumber\\
&-[b(z,u,Du)]_{\lambda}\cdot\tau^{h}(\eta^2(x)\xi^2(t)\chi_{\varepsilon}(t)\tau^{-h}u_{\lambda}(x,t))dxdt=0.
\end{align}
We first note that
\begin{equation}\label{hd3.10}
\partial_t(|\tau^{-h}u_{\lambda}|^2\eta^2\xi^2\chi_{\varepsilon})
=2\tau^{-h}u_{\lambda}\partial_t(\tau^{-h}u_{\lambda})\eta^2\xi^2\chi_{\varepsilon}
 +2|\tau^{-h}u_{\lambda}|^2\eta^2\xi'\xi\chi_{\varepsilon}+|\tau^{-h}u_{\lambda}|^2\eta^2\xi^2\chi_{\varepsilon}',
\end{equation}
and
\begin{align*}
  H_1:= & \int_{-T}^{0}\int_{\Omega}\partial_tu_{\lambda}\cdot\tau^{h}(\eta^2(x)\xi^2(t)\chi_{\varepsilon}(t)\tau^{-h}u_{\lambda}(x,t))dxdt\\
=& \int_{-T}^{0}\int_{\Omega}\partial_t(\tau^{-h}u_{\lambda})\cdot(-\eta^2(x)\xi^2(t)\chi_{\varepsilon}(t)\tau^{-h}u_{\lambda}(x,t))dxdt.
\end{align*}
This, combined with (\ref{hd3.10}) implies that
\begin{align*}
H_1= & \frac{1}{2} \int_{-T}^{0}\int_{\Omega}|\tau^{-h}u_{\lambda}|^2\eta^2\xi^2\chi_{\varepsilon}'dxdt
+\int_{-T}^{0}\int_{\Omega}|\tau^{-h}u_{\lambda}|^2\eta^2\xi'\xi\chi_{\varepsilon}\\
&-\frac{1}{2} \int_{-T}^{0}\int_{\Omega}\partial_t(|\tau^{-h}u_{\lambda}|^2\eta^2\xi^2\chi_{\varepsilon})dxdt\\
:= &H_2+H_3+H_4.
\end{align*}
Recalling   the definition of $\chi_{\varepsilon}$, we can find that
\begin{equation*}
H_2=-\frac{1}{2}\Xint-_{l}^{l+\varepsilon}\int_{\Omega}|\tau^{-h}u_{\lambda}|^2\eta^2\xi^2dxdt.
\end{equation*}
Moreover, note  that $H_4=0$. Then, by passing to the limit for $\lambda\longrightarrow0$ in (\ref{hd3.9}), we obtain
\begin{align}\label{hd3.11}
&-\frac{1}{2}\Xint-_{l}^{l+\varepsilon}\int_{\Omega}|\tau^{-h}u|^2\eta^2\xi^2dxdt+\int_{Q_T}|\tau^{-h}u|^2\eta^2\xi'\xi\chi_{\varepsilon}dxdt \nonumber\\
&+ \int_{Q_T}a(z,u,Du)\cdot D\tau^{h}(\eta^2(x)\xi^2(t)\chi_{\varepsilon}(t)\tau^{-h}u(x,t))dxdt\nonumber\\
&-\int_{Q_T}b(z,u,Du)\cdot \tau^h(\eta^2\xi^2\chi_{\varepsilon}(t)\tau^{-h}u) dxdt:=H_2+H_3+H_5+H_6=0.
\end{align}
Observe that
\begin{align*}
H_5= & \int_{Q_T}a(z,u,Du)\cdot \tau^{h}(2\eta\nabla\eta\otimes\xi^2\chi_{\varepsilon}\tau^{-h}u+\eta^2\xi^2\chi_{\varepsilon}\tau^{-h}Du)dxdt \\
= & \int_{Q_T}a(z,u,Du)\cdot \tau^{h}(2\eta\nabla\eta\otimes\xi^2\chi_{\varepsilon}\tau^{-h}u)dxdt
-\int_{Q_T}\tau^{-h}[a(z,u,Du)]\cdot(\eta^2\xi^2\chi_{\varepsilon}\tau^{-h}Du)dxdt\\
:=&H_{51}-H_{52},
\end{align*}
and
\begin{align*}
 H_{52}&= \int_{Q_T}[\tau^{-h}_3a((x,t-h),u(x,t-h),Du)+\tau^{-h}_{1,2}a(z,u,Du)]\cdot(\eta^2\xi^2\chi_{\varepsilon}\tau^{-h}Du)dxdt \\
 &:= H_{521}+ H_{522}.
\end{align*}
By the aid of (\ref{hd2.5}), there holds
\begin{equation}\label{hd3.12}
 H_{521}\geq c_1\int_{-T}^{0}\int_{\Omega}|\tau^{-h}Du|^2\eta^2\xi^2\chi_{\varepsilon}dxdt,
\end{equation}
where $c_1$ depends on $\nu$. Furthermore, applying ($H_3$), (\ref{x2.3}) and Young's inequality, it holds that
\begin{equation}\label{hd3.011}
|H_{522}|\leq \frac{c_1}{2}\int_{-T}^{0}\int_{\Omega}|\tau^{-h}Du|^2\eta^2\xi^2\chi_{\varepsilon}dxdt+ch^2\int_{Q_T}(1+|Du|^2)dxdt.
\end{equation}
For the term $H_{51}$, making  use of (\ref{hd3.7}), we obtain
\begin{align}\label{hd3.13}
|H_{51}|\leq& \left|\int_{Q_T}a(z,u,Du)\cdot (2\eta\nabla\eta\otimes\xi^2(t+h)\tau^{h}(\tau^{-h}u))dxdt\right|\nonumber\\
&+\left|\int_{Q_T}a(z,u,Du)\cdot (2\eta\nabla\eta\otimes\tau^{h}(\xi^2\chi_{\varepsilon})\tau^{-h}u)dxdt\right| \nonumber\\
\leq & c\int_{\tilde{t}_0}^{\tilde{t}_1+h}\int_{supp(\eta)}(1+|Du|)\|\nabla\eta\|_{L^{\infty}_x}|\tau^{h}(\tau^{-h}u)|dxdt\nonumber\\
&+c\int_{\tilde{t}_0}^{\tilde{t}_1+h}\int_{supp(\eta)}(1+|Du|)\|\nabla\eta\|_{L^{\infty}_x}|\tau^{-h}u|dxdt\nonumber\\
\leq & c\left(\int_{\tilde{t}_0}^{\tilde{t}_1+h}\int_{supp(\eta)}(1+|Du|^2)dxdt\right)^{\frac{1}{2}}
\left(\int_{\tilde{t}_0}^{\tilde{t}_1+2h}\int_{supp(\eta)}|\tau^{-h}u|^2dxdt\right)^{\frac{1}{2}}\nonumber\\
\leq &c|h|^{\frac{1}{2}} \int_{Q_T}(1+|Du|^2)dxdt,
\end{align}
with $c=c(L,dist(supp(\eta),\partial\Omega))$.

Concerning $H_6$, in virtue of (\ref{4.96}) and (\ref{hd3.7}), we have
\begin{align}\label{hd3.16}
H_6&\leq\left(\int_{Q_T}|b(z,u,Du)|^2dxdt\right)^{\frac{1}{2}}
\left(\int_{Q_T}|\tau^h(\eta^2\xi^2\chi_{\varepsilon}(t)\tau^{-h}u)|^2dxdt\right)^{\frac{1}{2}}\nonumber\\
&\leq c|h|^{\frac{1}{2}} \int_{Q_T}(1+|Du|^2)dxdt,
\end{align}
where $c=c(L,dist(supp(\eta),\partial\Omega))$.

Note that $H_2\leq0$, combining (\ref{hd3.11})-(\ref{hd3.16}) and (\ref{hd3.7}),  we finally obtain
\begin{align*}
\int_{-T}^{0}\int_{\Omega} |\tau^{-h}Du|^2\eta^2\xi^2\chi_{\varepsilon}dxdt\leq& c|h|^{\frac{1}{2}}\int_{Q_T}(1+|Du|^2)dxdt
+c \int_{Q_T} |\tau^{-h}u|^2\eta^2\xi\xi'\chi_{\varepsilon}dxdt\\
\leq &  c|h|^{\frac{1}{2}}\int_{Q_T}(1+|Du|^2)dxdt+c|h|\int_{Q_T}(1+|Du|^2)dxdt\\
\leq &  c|h|^{\frac{1}{2}}\int_{Q_T}(1+|Du|^2)dxdt,
\end{align*}
where $c=c(L,\nu, dist (supp(\eta),\partial\Omega))$.

Applying the property of $\xi(t)$ and $\eta(x)$, letting $\varepsilon\longrightarrow0$ and $l\longrightarrow0$, then we have (\ref{hd3.8}).
\end{proof}
According to (\ref{hd3.8}), we can re-estimate the term  $I_2$ in (\ref{hd3.4}):
\begin{align}\label{hd3.18}
|I_2| \leq&\left(\int_{t_0}^{t_1}\int_{\Omega}\eta^2|[a(z,u,Du)]_h|^2dxdt\right)^{\frac{1}{2}}
\left(\int_{t_0}^{t_1}\int_{\Omega}\eta^2|\tau^{h}Du|^2dxdt\right)^{\frac{1}{2}} \nonumber\\
\leq& c|h|^{\frac{1}{4}}\|\eta\|_{L^{\infty}}^{2}\int_{Q_T}(1+|Du|^2)dxdt.
\end{align}
From (\ref{hd3.18}) and note that $|h|<1$, then we can rewrite (\ref{hd3.7}) as  follows
\begin{equation}\label{hd3.19}
\int_{t_0}^{t_1}\int_{\Omega'}|\tau^hu(x,t)|^{2}dxdt\leq c_2|h|^{1+\frac{1}{4}}\int_{Q_T}(1+|Du|^2)dxdt,
\end{equation}
where $c_2=c_2(L,\nu, dist (supp(\eta),\partial\Omega))$.

Similarly, in the proof of Lemma \ref{hdle3.2}, we can use (\ref{hd3.19}) instead of (\ref{hd3.7}), then we arrive at
\begin{equation}\label{hd3.20}
 \int_{t_0}^{t_1}\int_{\Omega}|\tau^hDu(x,t)|^{2}\eta^2dxdt\leq c_2|h|^{\frac{1}{2}(1+\frac{1}{4})}\int_{Q_T}(1+|Du|^2)dxdt.
\end{equation}
Like the estimation of (\ref{hd3.18})-(\ref{hd3.20}), by iteration argument, we finally obtain
\begin{equation}\label{hd3.21}
 \int_{t_0}^{t_1}\int_{\Omega}|\tau^hDu(x,t)|^{2}\eta^2dxdt\leq c_k|h|^{a_k}\int_{Q_T}(1+|Du|^2)dxdt,
\end{equation}
with $a_k=\frac{1}{2}(1+\frac{a_{k-1}}{2})$ and  $a_0=\frac{1}{2}$, $c_k=c_k(L,\nu, dist
(supp(\eta),\partial\Omega))$. Letting $k\longrightarrow\infty$, we have $a_k\longrightarrow\frac{2}{3}$, applying the property of $\eta(x)$, then we
obtain (\ref{hd2.6}).
\subsection{Fractional space differentiability}
In this subsection,  we propose to deduce the fractional space differentiability of the gradient of weak solution $u$ to (\ref{1.1}). To begin with, we
discuss the general case for $p\geq2$.
\begin{Lemma}
Let $\Omega\subset\mathbb{R}^n$ $(n=2,3)$ be a bounded  domain. Let $u\in L^{\infty}(-T,0;L^2(\Omega;\mathbb{R}^N))\cap L^p(-T,0;$
$W^{1,p}(\Omega;\mathbb{R}^N))$ $(p\geq2)$ be a weak solution to (\ref{1.1}). Then, we have
\begin{equation}\label{hd4.1}
Du\in L_{loc}^2(-T,0;W^{\theta;2}_{loc}(\Omega)),
\end{equation}
and for any $Q_{2\rho}(z_0)\subset  Q_T$, there holds
\begin{align}\label{hd4.2}
 &\lim_{h\longrightarrow0}\sup_{t_0-(\rho/2)^2<t<t_0} \int_{B_{\rho/2}(x_0)}\frac{|\tau_hu(x,t_0)|^2}{|h|^{2\theta}}dx
 +\lim_{h\longrightarrow0} \int_{Q_{\rho/2}(z_0)}\frac{|\tau_h^3a(z,u,Du)|^2}{|h|^{2\theta}}dxdt\nonumber\\
 & \leq\frac{c}{\rho^2}\int_{Q_{2\rho}(z_0)}(1+|Du|^2)^{\frac{p}{2}}dxdt
 +\frac{c}{\rho^2}\int_{Q_{2\rho}(z_0)}|Du|^2dxdt+c\int_{t_0-(2\rho)^2}^{0}(1+\|u\|_{W^{1,p}(B_{2\rho}(x_0))}^{2})dt,
\end{align}
where $\theta\in(0,\frac12]$ and  $c=c(n,p,L,\nu)$.
\end{Lemma}
\begin{proof}
First,   replace the test function  $\varphi$ in (\ref{2.1}) by $\tau_{-h}\varphi$ with $0<|h|<\min\{dist(supp(\varphi),\partial Q_T),1,\rho\}$, then we infer that
\begin{equation*}
\int_{Q_T}\tau_hu\cdot\partial_t\varphi-\tau_ha(z,u,Du)\cdot D\varphi+\tau_hb(z,u,Du)\cdot\varphi dxdt=0.
\end{equation*}
By approximation, we choose $\varphi\equiv \phi\tau_hu$ in the previous equation with $\phi\in C_0^{\infty}(Q_T)$. Thus, we  are in a position to obtain
that
\begin{align}\label{hd4.4}
&-\frac{1}{2}\int_{Q_T}|\tau_hu|^2\partial_t\phi dxdt+\int_{Q_T}\phi\tau_h a(z,u,Du)\cdot D\tau_h udxdt\nonumber\\
&=-\int_{Q_T}\tau_h a(z,u,Du)\cdot (\nabla\phi\otimes \tau_h u)dxdt
+\int_{Q_T}\phi\tau_hb(z,u,Du)\cdot\tau_h udxdt.
\end{align}
Now, let us choose $\phi(x,t)=\bar{\chi}(t)\chi(t)\psi^2(x)$ with $\chi\in W^{1,\infty}((-T,0))$, $\chi(-T)=0$, $\partial_t\chi\geq0$ and
$0\leq\chi\leq1$,  $\psi\in C_0^{\infty}(B_{\rho}(x_0))$, $0\leq\psi\leq1$ and $\bar{\chi}:(-T,0)\longrightarrow \mathbb{R}$ is a Lipschitz continuous
function, defined by
\begin{align*}
\bar{\chi}(t)=
\begin{cases}
1 &t\leq t_0 \\
affine & t_0<t<t_0+\delta,\\
0&t\geq t_0+\delta,
\end{cases}
\end{align*}
with $-T<t_0<t_0+\delta<0$,  $\delta>0$. Thus, letting $\delta\longrightarrow0$, then (\ref{hd4.4}) becomes
\begin{align}\label{hd4.5}
\frac{1}{2}&\int_{B_{\rho}(x_0)}\chi(t_0)\psi^2(x)|\tau_hu(x,t_0)|^2 dx+\int_{Q_{t_0}}\chi\psi^2\tau_ha(z,u,Du)\cdot D(\tau_h u)dxdt\nonumber\\
=&-2\int_{Q_{t_0}}\chi\psi \tau_ha(z,u,Du)\cdot(\nabla\psi\otimes\tau_hu)dxdt+\int_{Q_{t_0}}\tau_hb(z,u,Du)\cdot\tau_h u\chi(t)\psi^2 (x)dxdt\nonumber\\
&+\frac{1}{2}\int_{Q_{t_0}}(\partial_t\chi)\psi^2(x)|\tau_hu|^2dxdt,
\end{align}
where  $Q_{t_0}:=B_{\rho}(x_0)\times (-T,t_0)$.

Observing that
\begin{equation*}
\tau_ha(z,u,Du)=\tau_h^{1,2}a(z,u,Du(x+he_i,t))+\tau_h^{3}a(z,u,Du),
\end{equation*}
and
\begin{equation*}
  \tau_h^3 a(z,u,Du)=\int_0^1\frac{\partial a(z,u,Du+s\tau_h Du)}{\partial F}ds\cdot\tau_h Du.
\end{equation*}
Then, applying  ($H_2$) and Lemma \ref{le2.1},  we can find that
\begin{align*}
\tau_h^3 a(z,u,Du)\cdot D(\tau_h u)\geq & \nu\int_0^1(1+|Du(x,t)+s\tau_hDu(x,t)|^2)^{\frac{p-2}{2}}ds|\tau_hDu(x,t)|^2 \\
\geq& \frac{\nu}{c(p)}(1+|Du(x,t)|^2+|Du(x+he_i,t)|^2)^{\frac{p-2}{2}}|\tau_hDu(x,t)|^2\\
\equiv&\frac{\nu}{c(p)} D(h)(x,t)^{\frac{p-2}{2}}|\tau_hDu(x,t)|^2.
\end{align*}
Here we have used abbreviated notation
\begin{equation*}
D(h)(x,t):=1+|Du(x,t)|^2+|Du(x+he_i,t)|^2.
\end{equation*}
Thus, we conclude that
\begin{equation}\label{hd4.6}
\int_{Q_{t_0}}\chi\psi^2\tau_h^3a(z,u,Du)\cdot D(\tau_hu)dxdt\geq\frac{\nu}{c(p)}\int_{Q_{t_0}}\chi\psi^2D(h)(x,t)^{\frac{p-2}{2}}|\tau_hDu(x,t)|^2dxdt.
\end{equation}
Moreover, from ($H_3$) and (\ref{x2.3}), it follows that
\begin{equation}\label{hd4.05}
\int_{Q_{t_0}}\chi\psi^2\tau_h^{1,2}a(z,u,Du)\cdot D(\tau_hu)dxdt\leq c|h|\int_{Q_{t_0}}(1+|Du|^p)dxdt.
\end{equation}
Using ($H_1$), we further obtain
\begin{align}\label{hd4.7}
 |\tau_h^3 a(z,u,Du)|\leq &L\int_0^1(1+|Du(x,t)+s\tau_hDu(x,t)|^2)^{\frac{p-2}{2}}ds|\tau_hDu(x,t)| \nonumber\\
\leq &c(p)L(1+|Du(x,t)|^2+|Du(x+he_i,t)|^2)^{\frac{p-2}{2}}|\tau_hDu(x,t)|\nonumber\\
=&c(p)L D(h)(x,t)^{\frac{p-2}{2}}|\tau_hDu(x,t)|.
\end{align}
Therefore, applying (\ref{hd4.7}) and the Young's inequality, we have
\begin{align}\label{hd4.8}
&\left|2\int_{Q_{t_0}}\chi\psi \tau_h^3 a(z,u,Du)\cdot(\nabla\psi\otimes \tau_hu)dxdt\right| \nonumber\\
&\leq c(p)L\int_{Q_{t_0}}\chi\psi D(h)(x,t)^{\frac{p-2}{2}}|\tau_h Du|\cdot|\nabla \psi|\cdot|\tau_hu|dxdt\nonumber\\
&\leq \varepsilon\int_{Q_{t_0}}\chi\psi^2D(h)(x,t)^{\frac{p-2}{2}}|\tau_h Du|^2dxdt+c(\varepsilon)\int_{Q_{t_0}}\chi|\nabla\psi|^2D(h)(x,t)^{\frac{p-2}{2}}|\tau_h u|^2dxdt,
\end{align}
where $c(\varepsilon)=c(p)L^2\varepsilon^{-1}$.

Furthermore, in virtue of ($H_3$), it holds that
\begin{equation}\label{hd4.07}
\int_{Q_{t_0}}\chi\psi \tau_h^{1,2} a(z,u,Du)\cdot(\nabla\psi\otimes \tau_hu)dxdt
\leq c(L)\int_{Q_{t_0}}\chi\psi(1+|Du|)^{p-1}|\nabla\psi||\tau^hu|dxdt
\end{equation}
Inserting (\ref{hd4.6})-(\ref{hd4.07}) into (\ref{hd4.5}), and choosing $\varepsilon=\frac{\nu}{2 c(p)}$, then we  obtain
\begin{align}\label{hd4.9}
&\frac{c(p)}{\nu}\int_{B_{\rho}(x_0)}\chi(t_0)\psi^2(x) |\tau_hu(x,t_0)|^2dx+\int_{Q_{t_0}}\chi\psi^2D(h)(x,t)^{\frac{p-2}{2}}|\tau_h Du|^2dxdt\nonumber\\
\leq& c\int_{Q_{t_0}}\chi|\nabla \psi|^2D(h)(x,t)^{\frac{p-2}{2}}|\tau_h u|^2+(\partial_t\chi)\psi^2|\tau_hu|^2dxdt
+ c|h|\int_{Q_{t_0}}(1+|Du|^p)dxdt\nonumber\\
&+c\int_{Q_{t_0}}\chi\psi(1+|Du|)^{p-1}|\nabla\psi||\tau^hu|dxdt+\int_{Q_{t_0}}\tau_hb(z,u,Du)\cdot(\tau_hu)\chi(t)\psi^2(x)dxdt,
\end{align}
where $c=c(n,L,p,\nu)$.

Let $Q_{\rho}(z_0)\subset Q_T$ with $\rho$ suitable small such that $2\rho< \min\{dist(x_0,\partial\Omega),1\}$. Now, we choose
 $\chi\in W^{1,\infty}(\mathbb{R})$ and $\psi\in C_0^{\infty}(B_{\rho}(x_0))$ satisfying
\begin{align*}
\begin{cases}
 \chi(t)=0 &t\in (-\infty,t_0-\rho^2), \\
   \chi(t)=1& t\in(t_0-(\rho/2)^2,\infty),\\
   0\leq\partial_t\chi\leq(2/\rho)^2&t\in \mathbb{R},
  \end{cases}
\end{align*}
and
\begin{align*}
\begin{cases}
  \psi=1 & x\in B_{\rho/2} (x_0),\\
  0\leq\psi\leq1 & x\in B_{\rho} (x_0),\\
  |\nabla\psi|\leq\frac{3}{\rho}&x\in B_{\rho}(x_0).
  \end{cases}
\end{align*}
Note that, for any $\theta\in (0,\frac12]$, the right-most term of (\ref{hd4.9}) can be estimated as
\begin{align}\label{hd4.10}
&\int_{Q_{t_0}}\tau_hb(z,u,Du)\cdot(\tau_hu)\chi(t)\psi^2(x)dxdt\nonumber\\
&\leq c \int^0_{t_0-(2\rho)^2}\|\psi^2\tau_hu\|_{L^p(B_{\rho}(x_0))}(|B_{\rho}|^{\frac1{p'}}+\|u\|_{W^{1,p}(B_{\rho}(x_0))})dt\nonumber\\
&\leq c(L,|\Omega|)\int_{t_0-(2\rho)^2}^0\|\tau_hu\|^{p\theta}_{L^p(B_{\rho}(x_0))}(1+\|u\|_{W^{1,p}(B_{\rho}(x_0))}^{2-p\theta}).
\end{align}
Employing (\ref{hd4.9}), (\ref{hd4.10}), we finally obtain
\begin{align*}
& \int_{B_{\rho}(x_0)}\chi(t_0)\psi^2(x) |\tau_hu(x,t_0)|^2dx+\int_{Q_{t_0}}\chi\psi^2D(h)(x,t)^{\frac{p-2}{2}}|\tau_h Du|^2dxdt \\
\leq &c\int_{Q_{t_0}}\frac{1}{\rho^2}\chi D(h)(x,t)^{\frac{p-2}{2}}|\tau_h u|^2+\frac{1}{\rho^2}\psi^2(x)|\tau_h u|^2dxdt
+ c|h|\int_{Q_{t_0}}(1+|Du|^p)dxdt\\
&+\frac{c}{\rho}\int_{Q_{t_0}}\chi\psi(1+|Du|)^{p-1}|\tau^hu|dxdt
+c\int^0_{t_0-(2\rho)^2}\|\tau_{h,k}u\|^{p\theta}_{L^p(B_{\rho}(x_0))}(1+\|u\|^{2-p\theta}_{W^{1,p}(B_{\rho}(x_0))})dt
\end{align*}
By Lemma \ref{hdle2.1}, the previous inequality implies that
\begin{align}\label{hd4.14}
& \int_{B_{\rho}(x_0)}\chi(t_0)\psi^2(x) |\tau_hu(x,t_0)|^2dx+\int_{Q_{t_0}}\chi\psi^2|\tau_h^3 a(z,u,Du)|^2dxdt\nonumber \\
 \leq& c\int_{Q_{t_0}}\frac{1}{\rho^2}\chi D(h)(x,t)^{\frac{p-2}{2}}|\tau_h u|^2+\frac{1}{\rho^2}\psi^2(x)|\tau_h u|^2dxdt\nonumber \\
&+c\int^0_{t_0-(2\rho)^2}\|\tau_{h}u\|^{p\theta}_{L^p(B_{\rho}(x_0))}(1+\|u\|^{2-p\theta}_{W^{1,p}(B_{\rho}(x_0))})dt\nonumber \\
&+ c|h|\int_{Q_{t_0}}(1+|Du|^p)dxdt
+\frac{c}{\rho}\int_{Q_{t_0}}\chi\psi(1+|Du|)^{p-1}|\tau^hu|dxdt
\end{align}
Dividing both side in (\ref{hd4.14}) by $|h|^{2\theta}$. Then, we can see that
\begin{align*}
  \sup_{t_0-(\rho/2)^2<t<t_0} & \int_{B_{\rho/2}(x_0)}|\tau_hu(x,t_0)|^2|h|^{-2\theta}dx+\int_{Q_{\rho/2}(z_0)}|\tau_h^3 a(z,u,Du)|^2|h|^{-2\theta}dxdt\\
\leq &c\int_{Q_{t_0}}\frac{1}{\rho^2}\chi D(h)(x,t)^{\frac{p-2}{2}}|\tau_h u|^2|h|^{-2\theta}+\frac{1}{\rho^2}\psi^2(x)|\tau_h u|^2|h|^{-2\theta}dxdt\\
&+c\int^0_{t_0-(2\rho)^2}
\|\tau_{h}u\|^{p\theta}_{L^p(B_{\rho}(x_0))}|h|^{-2\theta}(1+\|u\|^{2-p\theta}_{W^{1,p}(B_{\rho}(x_0))})dt\\
&+ c|h|^{1-2\theta}\int_{Q_{t_0}}(1+|Du|^p)dxdt
+\frac{c}{\rho}\int_{Q_{t_0}}\chi\psi(1+|Du|)^{p-1}|\tau^hu|h^{-2\theta}dxdt\nonumber\\
\leq & c\int_{Q_{t_0}}\frac{1}{\rho^2}\chi D(h)(x,t)^{\frac{p-2}{2}}|\triangle_h u|^2+\frac{1}{\rho^2}\psi^2(x)|\triangle_h u|^2dxdt\\
&+c\int^0_{t_0-(2\rho)^2}
\|\tau_{h}u\|^{p\theta}_{L^p(B_{\rho}(x_0))}|h|^{-2\theta}(1+\|u\|^{2-p\theta}_{W^{1,p}(B_{\rho}(x_0))})dt\\
&+ c|h|^{1-2\theta}\int_{Q_{t_0}}(1+|Du|^p)dxdt
+\frac{c}{\rho}\int_{Q_{t_0}}\chi\psi(1+|Du|)^{p-1}|\tau^hu|^{1-2\theta}|\triangle_hu|^{2\theta}dxdt
\end{align*}
Using  the standard estimate for difference quotients, letting $|h|\longrightarrow 0$, then we have (\ref{hd4.2}) and (\ref{hd4.1}).
\end{proof}
\begin{proof}[Proof of Theorem \ref{th1.2}]
Taking into account Proposition \ref{pr2.1}, (\ref{hd2.6}) and Sec. \ref{se6.1} for $p=2$, we have proved the fractional time differentiability of $Du$.
Hence, it is remains to prove the fractional space differentiability of $Du$.
From (\ref{hd4.2})  we can see that for  $p=2$
\begin{align*}
 &\lim_{h\longrightarrow0}\sup_{t_0-(\rho/2)^2<t<t_0} \int_{B_{\rho/2}(x_0)}\frac{|\tau_hu(x,t_0)|^2}{|h|^{2\theta}}dx
 +\lim_{h\longrightarrow0} \int_{Q_{\rho/2}(z_0)}\frac{|\tau_hDu|^2}{|h|^{2\theta}}dxdt\nonumber\\
 & \leq\frac{c}{\rho^2}\int_{Q_{2\rho}(z_0)}(1+|Du|^2)dxdt+\frac{c}{\rho^2}\int_{Q_{2\rho}(z_0)}|Du|^2dxdt
 +c\int_{t_0-(2\rho)^2}^{0}(1+\|u\|_{H^1(B_{2\rho}(x_0))}^{2})dt.
\end{align*}
This, combined with (\ref{hd2.7}) implies the fractional space differentiability of $Du$. Finally, applying Lemma \ref{hdle2.3} we obtain (\ref{hd1.7}).
Thus, we have completed the proof of Theorem \ref{th1.2}.
\end{proof}

\end{document}